\documentclass[a4paper,11pt]{amsart}
\usepackage{amssymb,amscd,amsxtra,xypic,longtable}
\usepackage[all]{xy}
\usepackage{color}
\usepackage{array}

\usepackage[pdftex]{hyperref}
\hypersetup{%
bookmarksnumbered=true,%
colorlinks=true,%
pagebackref=true, %
pdftitle={},%
pdfauthor={}}

\theoremstyle{plain}
\newtheorem{Thm}{Theorem}[section]
\newtheorem{Lem}[Thm]{Lemma}

\newtheorem{Prop}[Thm]{Proposition}

\theoremstyle{definition}
\newtheorem{Def}[Thm]{Definition}
\newtheorem{Def-Lem}[Thm]{Definition-Lemma}
\newtheorem{Cond}[Thm]{Condition}
\newtheorem{Rem}[Thm]{Remark}

\newtheorem*{Ack}{Acknowledgments}
\newtheorem{Ex}[Thm]{Example}

\newcommand{\Aut}{\operatorname{Aut}}

\newcommand{\codim}{\operatorname{codim}}

\newcommand{\Proj}{\operatorname{Proj}}
\newcommand{\prt}{\partial}
\newcommand{\Sing}{\operatorname{Sing}}
\newcommand{\Spec}{\operatorname{Spec}}

\newcommand{\Cl}{\operatorname{Cl}}

\newcommand{\Int}{\operatorname{Int}}
\newcommand{\bNE}{\operatorname{\overline{NE}}}
\newcommand{\Bs}{\operatorname{Bs}}

\newcommand{\Exc}{\operatorname{Exc}}
\newcommand{\mult}{\operatorname{mult}}

\newcommand{\lcm}{\operatorname{lcm}}

\newcommand{\ord}{\operatorname{ord}}
\newcommand{\lct}{\operatorname{lct}}
\newcommand{\Supp}{\operatorname{Supp}}

\newcommand{\QI}{\mathrm{QI}}
\newcommand{\EI}{\mathrm{EI}}
\newcommand{\wprod}{\operatorname{wp}}
\newcommand{\Diff}{\operatorname{Diff}}

\newcommand{\mbA}{\mathbb{A}}

\newcommand{\mbC}{\mathbb{C}}

\newcommand{\mbP}{\mathbb{P}}
\newcommand{\mbQ}{\mathbb{Q}}

\newcommand{\mbZ}{\mathbb{Z}}

\newcommand{\mcF}{\mathcal{F}}

\newcommand{\mcH}{\mathcal{H}}
\newcommand{\mcI}{\mathcal{I}}

\newcommand{\mcO}{\mathcal{O}}

\newcommand{\mcW}{\mathcal{W}}

\newcommand{\msp}{\mathsf{p}}
\newcommand{\msq}{\mathsf{q}}

\newcommand{\msi}{\mathsf{i}}
\newcommand{\msI}{\mathsf{I}}
\newcommand{\Ibr}{\mathsf{I}_{\mathsf{br}}}
\newcommand{\IdP}{\mathsf{I}_{\mathsf{dP}}}
\newcommand{\IF}{\mathsf{I}_{\mathsf{F}}}
\newcommand{\IFi}{\mathsf{I}_{\mathsf{F},(\mathrm{i})}}
\newcommand{\IFii}{\mathsf{I}_{\mathsf{F},(\mathrm{ii})}}
\newcommand{\IFiia}{\mathsf{I}_{\mathsf{F},(\mathrm{iia})}}
\newcommand{\IFiib}{\mathsf{I}_{\mathsf{F},(\mathrm{iib})}}
\newcommand{\IFiic}{\mathsf{I}_{\mathsf{F},(\mathrm{iic})}}

\newcommand{\ratmap}{\dashrightarrow}

\makeatletter
\def\imod#1{\allowbreak\mkern10mu({\operator@font mod}\,\,#1)}
\makeatother

\title[Alpha invariants of birationally bi-rigid Fano 3-folds I]{Alpha invariants of birationally bi-rigid\\ Fano 3-folds I}
\author{In-Kyun~Kim \and Takuzo~Okada \and Joonyeong~Won}
\address{Research institute of Mathematics, Seoul National University, Seoul 08826, Korea}
\email{soulcraw@gmail.com}
\address{Department of Mathematics, Faculty of Science and Engineering, Saga University, Saga 840-8502 Japan}
\email{okada@cc.saga-u.ac.jp}
\address{Center for Geometry and Physics, Institute for Basic Science (IBS) 77 Cheongam-ro, Nam-gu, Pohang, Gyeongbuk, 37673, Korea}
\email{leonwon@kias.re.kr}
\date{}

\begin{document}
%%%%%%%%%%%%%%%%%%%%%%%%%%%%%%%%%%
%%%%%%%%%%%%%%%%%%%%%%%%%%%%%%%%%%

\begin{abstract}
We compute global log canonical thresholds of certain birationally bi-rigid  Fano $3$-folds embedded in weighted projective spaces as complete intersections of codimension $2$ and prove that they admit an orbifold K\"ahler--Einstein metric and are K-stable.
As an application, we give examples of super-rigid affine Fano $4$-folds.
\end{abstract}

\maketitle

%\tableofcontents

%%%%%%%%%%%%%%%%%%%%%%%%%%%%%%%%%%
%%%%%%%%%%%%%%%%%%%%%%%%%%%%%%%%%%
\section{Introduction} \label{sec:intro}
%%%%%%%%%%%%%%%%%%%%%%%%%%%%%%%%%%
%%%%%%%%%%%%%%%%%%%%%%%%%%%%%%%%%%

Throughout the article, the ground field is assumed to be the field of complex numbers.
The alpha invariant, which is also known as the global log canonical threshold, of a Fano variety is an invariant which measures singularities of pluri-anticanonical divisors.
The most important application of this invariant is as follows: If a Fano variety $X$ with only quotient singularities satisfies $\lct (X) > \dim X/(\dim X + 1)$, where $\lct (X)$ denotes the global log canonical threshold of $X$ (see Definition \ref{def:glct}), then $X$ admits an orbifold K\"ahler--Einstein metric (see \cite{DK01}, \cite{Na90}, \cite{Ti87}) and is K-stable (see \cite{OS}).

%Global log canonical thresholds have been intensively computed for (log) del Pezzo surfaces (see ) and for smooth Fano 3-folds (see).

We recall known results for Fano 3-folds which are complete intersections in weighted projective spaces.
In this paper, a weighted hypersurface in $\mbP (a_0,\dots,a_4)$ defined by an equation of degree $d$ that is quasi-smooth, well formed, has only terminal singularities and $\sum_{i=0}^4 a_i - d = 1$ is simply called a {\it Fano $3$-fold weighted hypersurface of index $1$}.
It is known that they form $95$ families (see \cite{IF} and \cite{CCC}) and their global log canonical thresholds have been computed in \cite{CPW}, \cite{Ch08}, \cite{Ch09b} and \cite{Ch09a}.

\begin{Thm}[{\cite[Corollary 1.45]{Ch09a}}] \label{thm:codim1br}
Let $X$ be a general Fano $3$-fold weighted hypersurface of index $1$ $($i.e.\  a general member of one of the $95$ families$)$.
Then $X$ admits an orbifold K\"ahler--Einsten metric and is K-stable.
\end{Thm}

We note that it is proved in \cite{CP} and \cite{CPR} that a Fano $3$-fold weighted hypersurface of index $1$ is birationally rigid, that is, it has a unique (up to isomorphism) structure of a Mori fiber space in its birational equivalence class.

A codimension $2$ weighted complete intersection of type $(d_1, d_2)$ in a weighted projective $5$-space $\mbP (a_0,\dots, a_5)$ that is quasi-smooth, well formed, has only terminal singularities and $\sum_{i=0}^5 a_i - (d_1 + d_2) = 1$ is simply called a {\it codimension $2$ Fano $3$-fold WCI $($weighted complete intersection$)$ of index $1$}.
By \cite{IF} and \cite{CCC}, the codimension $2$ Fano $3$-fold WCIs of index $1$ are known to form $85$ families which are specified as family No.~$\msi$ with $\msi \in \msI := \{1, 2, \cdots, 85\}$.
In \cite{Okada1}, birational geometry of these Fano $3$-folds are studied and the $85$ families are divided into the disjoint union of $3$ pieces $\msI = \msI_{\mathsf{br}} \cup \IF \cup \IdP$ with $|\Ibr| = 19$, $|\IF| = 60$ and $|\IdP| = 6$ according to their birational properties.
It is proved in \cite{AZ} and \cite{Okada1} that a codimension $2$ Fano $3$-fold WCI is birationally rigid if and only if it belongs to family No.~$\msi$ for some $\msi \in \Ibr$.
We have the following result for these birationally rigid varieties.

\begin{Thm}[{\cite[Theorem 1.5]{KOW}, \cite[Theorem 1.3]{Zhuang}}] \label{thm:codim2br}
Let $X$ be a codimension $2$ Fano $3$-fold WCI of index $1$ which is a general member of family No.~$\msi \in \Ibr$.
Then $X$ admits an orbifold K\"ahler--Einstein metric and is K-stable. 
\end{Thm}

We note that it is conjectured in \cite{KOW} that a birationally rigid Fano variety (of Picard number $1$) is K-stable.
See \cite{OO} and \cite{SZ} for partial results concerning this conjecture.

By Theorems \ref{thm:codim1br} and \ref{thm:codim2br}, K-stability and the existence of K\"ahler--Einstein metric of birationally rigid quasi-smooth and well formed Fano 3-fold weighted complete intersections of index $1$ (of any codimension) are proved under a generality assumption.

The aim of this article is to work with codimension $2$ Fano $3$-fold WCIs of index $1$ which are close to be but not birationally rigid.
Let $X$ be a member of family No.~$\msi \in \IF$.
Then, it is proved that there is a Fano $3$-fold $X'$ of Picard number $1$, which is birational but not isomorphic to $X'$.
It is believed that $X$ is {\it birationally bi-rigid} which means that a Mori fiber space which is birational to $X$ is isomorphic to either $X$ or $X'$.
In fact, in \cite{Okada2, Okada3}, birational bi-rigidity is proved for many families No.~$\msi \in \IF$.
We can still expect that $\lct (X)$ is large and the main aim of this article is to confirm this for $X$ with relatively small anti-canonical degree $(-K_X)^3$.

\begin{Thm} \label{mainthm1}
Let $X$ be a codimension $2$ Fano $3$-fold WCI of index $1$ which is a general member of family No.~$\msi$ with
\[
\msi \in \{42, 55, 66, 68, 69, 77, 79, 81, 80, 82, 83\} \subset \IF.
\]
Then $\lct (X) = 1$.
In particular $X$ is K-stable and admits an orbifold K\"{a}hler--Einstein metric. 
\end{Thm}

We refer readers to Section \ref{sec:deffam} for detailed information on the families above.
We discuss an application of Theorem \ref{mainthm1} in Section \ref{sec:ex} and give examples of super-rigid affine Fano $4$-folds, where the notion of super-rigid Fano variety is introduced by Cheltsov, Dubouloz and Park \cite{CDP}.

\begin{Ack}
The first author was supported by the National Research Foundation of Korea(NRF)
grant funded by the Korea government(MSIP) (No. NRF - 2017R1C1B1011921). 
The second author is partially supported by JSPS KAKENHI Grant Number JP18K03216. The
third author was supported by IBS-R003-D1, Institute for Basic Science in Korea.
\end{Ack}
%%%%%%%%%%%%%%%%%%%%%%%%%%%%%%%%%%
%%%%%%%%%%%%%%%%%%%%%%%%%%%%%%%%%%
\section{Preliminaries} \label{sec:prelim}
%%%%%%%%%%%%%%%%%%%%%%%%%%%%%%%%%%
%%%%%%%%%%%%%%%%%%%%%%%%%%%%%%%%%%

%%%%%%%%%%%%%%%%%%%%%%%%%%%%%%%%%%
\subsection{Basic definitions} \label{sec:basicdefs}
%%%%%%%%%%%%%%%%%%%%%%%%%%%%%%%%%%

Let $X$ be a Fano variety, i.e.\ a normal projective $\mbQ$-factorial variety with at most terminal singularities such that $-K_X$ is ample.

\begin{Def} \label{def:glct}
Let $(X,D)$ be a pair, that is, $D$ is an effective $\mbQ$-divisor, and let $\msp \in X$ be a point.
We define the {\it log canonical threshold} (LCT, for short) of $(X,D)$ and the {\it log canonical threshold} of $(X,D)$ {\it at} $\msp$ to be the numbers
\[
\begin{split}
\lct (X,D) &= \sup \{\, c \mid \text{$(X, c D)$ is log canonical} \,\}, \\
\lct_{\msp} (X,D) &= \sup \{\, c \mid \text{$(X, c D)$ is log canonical at $\msp$} \,\},
\end{split}
\]
respectively.
We define
\[
\lct_{\msp} (X) = \inf \{\, \lct_{\msp} (X,D) \mid D \text{ is an effective $\mbQ$-divisor}, D \equiv -K_X \,\},
\]
and for a subset $\Sigma \subset X$, we define
\[
\lct_{\Sigma} (X) = \inf \{\, \lct_{\msp} (X) \mid \msp \in \Sigma \, \}.
\]
The number $\lct (X) := \lct_X (X)$ is called the {\it global log canonical threshold} (GLCT, for short) or the {\it alpha invariant} of $X$
\end{Def}

Note that the numerical equivalence $D_1 \equiv D_2$ between $\mbQ$-divisors on a Fano variety $X$ is equivalent to $\mbQ$-linear equivalence $D_1 \sim_{\mbQ} D_2$. 
The following relation between $\mult_{\msp} (D)$ and $\lct_{\msp} (X,D)$ are used to bound $\lct_{\msp} (X,D)$ from below.

\begin{Lem}[{\cite[8.10 Lemma]{Kol}}]
For a nonsingular point $\msp \in X$ and an effective $\mbQ$-divisor $D$ on $X$, the inequality
\[
\frac{1}{\mult_{\msp} (D)} \le \lct_{\msp} (X,D)
\]
holds.
\end{Lem}

The following fact is frequently used.

\begin{Rem} \label{rem:covex}
Let $\msp$ be a point on a normal variety $X$.
Let $D_1$ and $D_2$ be effective $\mbQ$-divisors on $X$.
If both $(X, D_1)$ and $(X, D_2)$ are log canonical at $\msp$, then the pair $(X, \lambda D_1 + (1-\lambda) D_2)$ is log canonical at $\msp$ for any $\lambda \in \mbQ$ such that $0 \le \lambda \le 1$.
In particular, if $\lct_{\msp} (X) < 1$, then there exists an irreducible effective $\mbQ$-divisor $D \sim_{\mbQ} -K_X$ such that $(X, D)$ is not log canonical at $\msp$.
Here a $\mbQ$-divisor is {\it irreducible} if its support $\Supp (D)$ is a prime divisor.
\end{Rem}

We recall the definition of isolating class which is introduced by Corti, Pukhlikov and Reid \cite{CPR}. 

\begin{Def}
Let $L$ be a Weil divisor class on a variety $V$ and $\msp \in V$ a nonsingular point.
We say that a closed subset $\Xi \subset X$ is a $\msp$-{\it isolating set} if $\msp \in \Xi$ is an isolated component of $\Xi$.
We say that a Weil divisor class $L$ is a $\msp$-{\it isolating class} if the base locus of $|\mcI_{\msp}^k (k L)|$ is a $\msp$-isolating set for some $k > 0$.
\end{Def}

We generalize the notion of isolating class.
This generalization will play an important role in the computations of LCTs at nonsingular points.

\begin{Def}
Let $\msp \in V$ be a nonsingular point and $\Gamma$ a closed subset of $V$ such that any irreducible component of $\Gamma$ contains $\msp$.
We say that a closed subset $\Xi \subset X$ is a $(\msp, \Gamma)$-{\it isolating set} if $\msp \in \Xi$ and any component of $\Xi$ passing through $\msp$ is contained in $\Gamma$. 
We say that a Weil divisor class $L$ is a $(\msp, \Gamma)$-{\it isolating class} if there exists a positive integer $k$ such that the base locus of $|\mcI^k_{\msp} (k L)|$ is a $(\msp, \Gamma)$-isolating class.
\end{Def}

Note that a $\msp$-isolating class is a $(\msp, \Gamma)$-isolating class for any $\Gamma \subset X$ containing $\msp$, and that $(\msp, \msp)$-isolating class is the usual $\msp$-isolating class.

Let $\msp \in V$ be a $3$-dimensional terminal quotient singularity of type $\frac{1}{r} (1, a, r-a)$, where $r$ and $a$ are coprime positive integers with $r > a$.
Let $\varphi \colon W \to V$ be the weighted blowup of $V$ at $\msp$ with weight $\frac{1}{r} (1,a,r-a)$.
By \cite{Kawamata}, $\varphi$ is the unique divisorial contraction centered at $\msp$ and we call $\varphi$ the {\it Kawamata blowup} of $V$ at $\msp$.
If we denote by $E$ the $\varphi$-exceptional divisor, then $E \cong \mbP (1,a,r-a)$ and we have
\[
K_W = \varphi^*K_V + \frac{1}{r} E,
\]
and
\[
(E^3) = \frac{r^2}{a (r-a)}.
\]

\begin{Def}
For a $3$-dimensional terminal quotient singularity $\msp \in V$ of type $\frac{1}{r} (1, a, r-a)$, we define
\[
\operatorname{wp} (\msp) =  a (r-a)
\]
and call it the {\it weight product} of $\msp \in V$.
\end{Def}

%%%%%%%%%%%%%%%%%%%%%%%%%%%%%%%%%%
\subsection{Methods}
%%%%%%%%%%%%%%%%%%%%%%%%%%%%%%%%%%

\subsubsection{Methods for computing LCT at a point}

We recall methods for computing LCTs and consider a generalization.

\begin{Lem}[{\cite[Lemma 2.5]{KOW}}] \label{lem:exclL}
Let $V$ be a normal projective $\mbQ$-factorial $3$-fold such that $-K_V$ is nef and big, and let $\msp \in V$ be either a nonsingular point or a terminal quotient singular point of index $r$ $($Below we set $r = 1$ when $\msp \in V$ is a nonsingular point$)$.
Suppose that there are prime divisors $S_1, S_2$ on $V$ with the following properties.
\begin{enumerate}
\item $S_1 \sim_{\mbQ} - c_1 V$ and $S_2 \sim_{\mbQ} - c_2 K_V$ for some positive $c_1, c_2 \in \mbQ$.
\item The pair $(X, \frac{1}{c_1} S_1)$ is log canonical at $\msp$.
\item The scheme-theoretic intersection $\check{\Gamma} := \rho_{\msp}^* S_1 \cap \rho_{\msp}^* S_2$ is an irreducible and reduced curve such that $0 < \mult_{\check{\msp}} (\check{\Gamma}) \le c_2$, where $\rho_{\msp}$ denotes the index one cover of an open neighborhood of $\msp \in V$ and $\check{\msp}$ is the preimage of $\msp$ via $\rho_{\msp}$.
\item $r c_1 c_2 (-K_V)^3 \le 1$.
\end{enumerate}
Then $\lct_{\msp} (X) \ge 1$.
\end{Lem}

\begin{Rem}
The assumption (2) in Lemma \ref{lem:exclL} is different from the following assumption
\begin{enumerate}
\item[(2)] If $r = 1$, then $\mult_{\msp} (S_1) \le c_1$, and if $r > 1$, then $\ord_{F} (S_1) \le c_1/r$, where $F$ is the exceptional divisor of the Kawamata blowup of $V$ at $\msp$.
\end{enumerate}
in \cite[Lemma 2.5]{KOW}.
The latter implies the former, but not vice versa.
However, in the proof of the statement $\lct_{\msp} (X) \ge 1$ in \cite{KOW}, the assumption (2) is used to show that $(X, \frac{1}{c_1} S_1)$ is log canonical.
The same remark applies to Lemma \ref{lem:exclG}(2) below.
\end{Rem}

In the rest of this subsection, let $X$ be a Fano $3$-fold with $\Cl (X) \cong \mbZ$ and we assume that $\Cl (X)$ is generated by $A := -K_X$.

\begin{Lem}[{\cite[Lemma 2.6]{KOW}}] \label{lem:exclG}
Let $\msp \in X$ be a nonsingular point.
Suppose that one of the following conditions is satisfied.
\begin{enumerate}
\item 
There is a $\msp$-isolating class $lA$ and distinct prime divisors $S_1 \sim_{\mbQ} c_1 A$, $S_2 \sim_{\mbQ} c_2 A$ containing $\msp$ such that $\max \{c_1,c_2\} l (A^3) \le 1$. 
\item There is a $\msp$-isolating class $l A$ and a prime divisor $S \sim_{\mbQ} c A$ containing $\msp$ such that $(X, \frac{1}{c} S)$ is log canonical at $ \msp$ and $c l (A^3) \le 1$.
\end{enumerate}
Then $\lct_{\msp} (X) \ge 1$.
\end{Lem}

The following is a generalization of Lemma \ref{lem:exclG}(2).

\begin{Lem} \label{lem:criwisol}
Let $\msp \in X$ be a nonsingular point.
Suppose that there are a prime divisor $S \sim_{\mbQ} c A$ passing through $\msp$ and a $(\msp, \Gamma)$-isolating class $l A$, where $\Gamma$ is an irreducible and reduced curve with $d := (A \cdot \Gamma)$ and $m := \mult_{\msp} (\Gamma)$, satisfying the following properties.
\begin{enumerate}
\item $(X, \frac{1}{c} S)$ is log canonical at $\msp$.
\item One of the following is satisfied.
\begin{enumerate}
\item $l d \le m$ and $c m (A^3) \le d$.
\item $l d > m$ and $c l (A^3) \le 1$.
\end{enumerate}
\end{enumerate}
Then $\lct_{\msp} (X) \ge 1$.
\end{Lem}

\begin{proof}
Suppose that $\lct_{\msp} (X) < 1$.
Then there is an irreducible effective $\mbQ$-divisor $D \sim_{\mbQ} A$ such that $(X, D)$ is not log canonical at $\msp$.
By (1), $\Supp (D) \ne S$ and we can write 
\[
D \cdot S = \gamma \Gamma + \Delta,
\]
where $\gamma \ge 0$ and $\Delta$ is an effective $1$-cycle such that $\Gamma \not\subset \Supp (\Delta)$.
We have
\begin{equation} \label{eq:criwisol1}
c (A^3) = A \cdot D \cdot S = \gamma (A \cdot \Gamma) + (A \cdot \Delta) \ge d \gamma.
\end{equation}
For a large $k > 0$, we have $\Bs |\mcI_{\msp}^k (k l A)| = \Gamma \cup \Theta$, where $\Theta \subset X$ is a closed subset disjoint from $\msp$, since $l A$ is a $(\msp,\Gamma)$-isolating class.
Write $\Delta = \Delta' + \Xi$, where $\Delta'$ and $\Xi$ are effective $1$-cycles such that $\Supp (\Delta') \subset \Supp (\Theta)$ and no component of $\Supp (\Xi)$ is contained in $\Theta$.
Then we can take $M \in |\mcI^k_{\msp} (k l A)|$ such that $M$ does not contain any component of $\Supp (\Xi)$.
Note that 
\[
\mult_{\msp} (\Xi) = \mult_{\msp} (D \cdot S) - \gamma \mult_{\msp} (\Gamma) > 1 - m \gamma.
\]
We have
\begin{equation} \label{eq:criwisol2}
k l (c (A^3) - d \gamma) = M \cdot (D \cdot S - \gamma \Gamma) = M \cdot \Delta \ge M \cdot \Xi > k (1 - m \gamma).
\end{equation}

Suppose that (2-a) is satisfied.
By taking into account the inequality $m - ld \ge 0$, the combination of \eqref{eq:criwisol1} and \eqref{eq:criwisol2} implies
\[
(m-ld) \frac{c}{d} (A^3) \ge (m-ld) \gamma > 1 - c l (A^3),
\]
which is equivalent to $c m (A^3) > d$.
This is a contradiction.

Suppose that (2-b) is satisfied.
By taking into account the inequality $l d - m > 0$, the inequality \eqref{eq:criwisol2} implies
\[
c l (A^3) - 1 > (l d - m) \gamma \ge 0.
\]
This is a contradition.
Therefore $\lct_{\msp} (X) \ge 1$.
\end{proof}

\begin{Lem}[{\cite[Lemma 2.8]{KOW}}] \label{lem:singptNE}
Let $\msp \in X$ be a terminal quotient singular point and $\varphi \colon Y \to X$ the Kawamata blowup at $\msp$.
Suppose that $(-K_Y)^2 \notin \Int \bNE (Y)$ and there exists a prime divisor $S$ on $X$ such that $\tilde{S} \sim_{\mbQ} - m K_Y$ for some $m > 0$, where $\tilde{S}$ is the proper transform of $S$ on $Y$.
Then $\lct_{\msp} (X) \ge 1$.
\end{Lem}

\subsubsection{Methods for computing isolating classes}

We explain methods for computing $\msp$-isolating and $(\msp, \Gamma)$-isolating classes, which is important when we apply Lemmas \ref{lem:exclG} and \ref{lem:criwisol}.

Let $V$ be a normal projective variety embedded in a weighted projective space $\mbP = \mbP (a_0,\dots,a_n)$ with homogeneous coordinates $x_0,\dots,x_n$ with $\deg x_i = a_i$, and let $A$ be a Weil divisor on $V$ such that $\mcO_V (A) \cong \mcO_V (1)$.
We do not assume that $a_0 \le \cdots \le a_n$.

\begin{Def}
Let $\msp \in V$ be a point and $\Gamma \subset V$ a closed subset such that  $\msp \in \Gamma$.
We say that (weighted) homogeneous polynomials $g_1, \dots, g_m \in \mbC [x_0,\dots,x_n]$ are $(\msp,\Gamma)$-{\it isolating polynomials} if 
\[
(g_1 = \cdots = g_m = 0) \cap V
\]
is a $(\msp,\Gamma)$-isolating set.
We say that $g_1, \dots, g_m$ are $\msp$-{\it isolating polynomials} if they are $(\msp,\msp)$-isolating polynomials.
\end{Def}

\begin{Lem} \label{lem:isolsetclass}
Let $\msp \in V$ be a point and $\Gamma \subset V$ a closed subset such that $\msp \in \Gamma$.
If $g_1, \dots, g_m \in \mbC [x_0,\dots,x_n]$ are $(\msp,\Gamma)$-isolating homogeneous polynomials, then $l A$ is a $(\msp, \Gamma)$-isolating class, where 
\[
l := \max \{\, \deg g_i \mid 1 \le i \le m, \}.
\]
In particular, if $\Gamma = \{\msp\}$ in a neighborhood of $\msp$, then $l A$ is a $\msp$-isolating class.
\end{Lem}

\begin{proof}
We may assume that $\deg g_m = \max \{\, \deg g_i \mid 1 \le i \le m \,\}$.
Let $d > 0$ be an integer such that $\deg g_i \mid d$ for any $i$ and set $k_i = d/\deg g_i$.
Let $\mcH \subset |d A|$ be the linear system spanned by
\[
g_1^{k_1}, \dots, g_m^{k_m}.
\]
Note that $d = k_m \deg g_m = k l$, where we set $k := k_m$.
Note also that $g_i^{k_i}$ vanishes along $\msp$ with multiplicity at least $k_i \ge k$, so that $\mcH \subset |\mcI^k_{\msp} (k l A)|$.
Then we have $\Bs (|\mcI^k_{\msp} (k lA)|) \subset \Bs (\mcH) = \Xi$ set-theoretically.
This shows that $l A$ is an $(\msp, \Gamma)$-isolating class.
\end{proof}

\begin{Lem}[{\cite[Lemma 3.5]{KOW}}] \label{lem:findisol}
Let $\pi \colon V \ratmap \mbP (a_0,\dots,a_m)$ be the projection by the coordinates $x_0,\dots,x_m$ and suppose that $\pi$ is a generically finite map onto the image.
Let $\Exc (\pi) \subset V$ denotes the locus contracted by $\pi$.
If $\msp \notin H_{x_j} \cup \Exc (\pi)$, where $0 \le j \le m$, then $lA$ is a $\msp$-isolating class, where
\[
l := \max_{0 \le k \le m, k \ne j} \{ \operatorname{lcm} (a_j,a_k)\}.
\]
\end{Lem}

\begin{proof}
We set $\msp = (\xi_0\!:\!\cdots\!:\!\xi_n)$.
Then $\pi (\msp) = (\xi_0\!:\!\cdots\!:\!\xi_m)$ and $\xi_j \ne 0$.
For $k \ne j$, we put $a'_k = a_k/\gcd (a_j,a_k)$ and $a'_j = a_j/\gcd (a_j,a_k)$.
It is easy to see that the common zero loci of the sections in
\[
\{g_0,\dots,\hat{g}_j,\dots,g_m\}, \ \text{where $g_k = \xi_j^{a'_k} x_k^{a'_j} - \xi_k^{a'_j} x_j^{a'_k}$},
\]
is $\{\pi (\msp)\}$.
It follows that the common zero loci on $V$ of the above set is the fiber $\pi^{-1} (\pi (\msp))$ which is a finite set of points since $\pi$ is generically finite and $\msp \notin \Exc (\pi)$.
This shows that $l A$, where $l = \max \{\, \deg g_i \mid 1 \le i \le m \,\}$, isolates $\msp$. 
\end{proof}

\begin{Rem} \label{rem:findisol}
Lemma \ref{lem:findisol} can be applied for $m = n$.
In that case the projection $\pi \colon V \ratmap \mbP (a_0,\dots,a_n)$ is the identity map and the assumption that $\pi$ is a finite morphism is automatically satisfied.
\end{Rem}

\begin{Def}
For positive integers $a, b$, we define 
\[
\langle a, b \rangle := \{\, m a + n b \mid m, n \in \mbZ, m \ge 0, n \ge 0, (m, n) \ne (0,0)\,\},
\] 
which is the semigroup generated by $a, b$.
\end{Def}

\begin{Lem} \label{lem:isolstr}
Let $\msp = (\alpha_0\!:\!\alpha_1\!:\!\cdots\!:\!\alpha_n) \in V$ be a point such that $\alpha_0 \ne 0$ and $\alpha_1 \ne 0$.
For $i = 1,\dots,n$, we set
\[
l_i := 
\begin{cases}
\operatorname{lcm} \{a_0,a_1\}, & \text{if $i = 1$}, \\
a_i \min \{\, k > 0 \mid k a_i \in \langle a_0,a_1 \rangle \,\}, & \text{if $i \ge 2$ and $\alpha_i \ne 0$}, \\
a_i, & \text{if $i \ge 2$ and $\alpha_i = 0$},
\end{cases}
\]
and
\[
l := \max \{l_1,\dots,l_n\}.
\]
Then $l A$ is a $\msp$-isolating class.
\end{Lem}

\begin{proof}
We may assume $\alpha_0 = 1$.
For $i = 1,2,\dots,n$, we set $m_i = l_i/a_i$.
For $i \ge 2$ such that $\alpha_i \ne 0$, there are positive integers $b_i, c_i$ such that $m_i a_i = a_0 b_i + a_1 c_i$.
For such $i$, we set
\[
M_i = \alpha_1^{c_i} x_i^{m_i} - \alpha_i^{m_i} x_0^{n_i} x_1^{c_i}.
\]
Then it is easy to see that the common zero locus of the elements in the set
\[
\{ \alpha_0^{m_0} x_1^{m_1} - \alpha_1^{m_1} x_0^{m_0} \} \cup \{\, M_i \mid \text{$i \ge 2$ and $\alpha_i \ne 0$} \,\} \cup \{\, x_i \mid \text{$i \ge 2$ and $\alpha_i = 0$} \,\}
\]
is a finite set of points (including $\msp$) and the maximum degree of those elements is $l$.
Thus $l A$ is a $\msp$-isolating class.
\end{proof}

%%%%%%%%%%%%%%%%%%%%%%%%%%%%%%%%%%
%%%%%%%%%%%%%%%%%%%%%%%%%%%%%%%%%%
\section{Definition of families and their birational geometry} \label{sec:deffam}
%%%%%%%%%%%%%%%%%%%%%%%%%%%%%%%%%%
%%%%%%%%%%%%%%%%%%%%%%%%%%%%%%%%%%

We recall basic definitions for weighted complete intersections.

\begin{Def}
Let $X$ be a complete intersection in a weighted projective space 
\[
\mbP := \mbP (a_0,\dots,a_N) = \Proj \mbC [x_0,\dots,x_N],
\] 
defined by equations $F_1 = \cdots = F_c = 0$, where $F_i \in \mbC [x_0,\dots,x_N]$ are homogeneous polynomial of degree $d$ with respect to the grading $\deg x_i = a_i$.
We say that $X$ is {\it quasi-smooth} if the quasi-affine cone
\[
C_X \subset \mbA^{N+1} = \Spec \mbC [x_0,\dots,x_N],
\]
which is defined by $F_1 = \cdots = F_c = 0$, is smooth outside the origin. 
We say that $X$ is {\it well formed} if $\codim (X \cap \Sing (\mbP)) \ge 2$ in $X$.
\end{Def}

%%%%%%%%%%%%%%%%%%%%%%%%%%%%%%%%%%
\subsection{Main objects}
%%%%%%%%%%%%%%%%%%%%%%%%%%%%%%%%%%

Let 
\[
X = X_{d_1, d_2} \subset \mbP (a_0,\dots,a_5) =: \mbP
\] 
be a codimension $2$ Fano $3$-fold WCI of index $1$ defined by two equations $F_1 = F_2 = 0$ with $\deg F_i = d_i$.
Recall that $X$ is quasi-smooth, well formed, has only terminal singularities and  $\sum_{i=0}^5 a_i - (d_1 + d_2) = 1$.
By quasi-smoothness, $X$ has only cyclic quotient singularities and that $\Cl (X) \cong \mbZ$ and it is generated by a divisor class $A$ on $X$ such that $\mcO_X (A) \cong \mcO_X (1)$ (see \cite[Remark 4.2]{Okada4}).
Since $X$ is well formed, the adjunction holds, that is, we have
\[
\mcO_X (K_X) \cong \mcO_X \left(d_1 + d_2  - \sum_{i=0}^{5} a_i \right) = \mcO_X (-1).
\]

As explained in the introduction, codimension $2$ Fano $3$-fold WCIs consists of $85$ families, that is, families No.~$\msi$ with $\msi \in \mathsf{I} := \{1,2,\dots,85\}$ and we have the division into disjoint 3 pieces $\mathsf{I} = \mathsf{I}_{\mathsf{br}} \cup \IF \cup \mathsf{I}_{\mathsf{dP}}$ with $|\mathsf{I}_{\mathsf{br}}| = 19$, $|\IF| = 60$ and $|\mathsf{I}_{\mathsf{dP}}| = 6$ according to their birational properties.

Our main objects are Fano $3$-folds indexed by $\IF$ and,  among them, we particularly consider subsets $\IFi \cup \IFii$ of $\IF$
whose detailed descriptions are given in Tables \ref{table:codim2Fanos-i} and \ref{table:codim2Fanos-ii} respectively (see also Remark \ref{rem:numcharact} below for their numerical characterizations).
In Tables \ref{table:codim2Fanos-i} and \ref{table:codim2Fanos-ii}, the anticanonical degree $(A^3) = (-K_X)^3$ and the basket of singularities are given in 3rd and 4th column respectively.
In the 4th column, the subscripts $\QI, \EI$ indicate that Kawamata blowup at the point leads to a birational involution which is called a quadratic involution or an elliptic involution respectively (see \cite[Section 5]{Okada1}).
The subscript $\mathrm{d}$ indicates that the point is a {\it distinguished singular point} (see Section \ref{sec:distsing} below). 

\begin{table}[htb]
\begin{center}
\caption{Families indexed by $\IFi$}
\label{table:codim2Fanos-i}
\begin{tabular}{clcc}
No. & $X_{d_1,d_2} \subset \mbP (a_0,\dots,a_5)$ & $(A^3)$ & Basket of singularities \\[0.5mm]
\hline \\[-3.5mm]
42 & $X_{10,12} \subset \mbP (1,1,4,5,6,6)$ & $1/6$ & $\frac{1}{2}$, $2 \times \frac{1}{6} (1,5)_{\mathrm{d}}$ \\[0.6mm]
55 & $X_{12,14} \subset \mbP (1,1,4,6,7,8)$ & $1/8$ & $\frac{1}{2}$, $\frac{1}{4}_{\QI}$, $\frac{1}{8} (1,7)_{\mathrm{d}}$ \\[0.6mm]
66 & $X_{14,15} \subset \mbP (1,2,5,6,7,9)$ & $1/18$ & $2 \times \frac{1}{2}$, $\frac{1}{6} (1,5)_{\QI}$, $\frac{1}{9} (2,7)_{\mathrm{d}}$ \\[0.6mm]
68 & $X_{14,15} \subset \mbP (1,3,5,6,7,8)$ & $1/24$ & $2 \times \frac{1}{3}$, $\frac{1}{6} (1,5)_{\EI}$, $\frac{1}{8} (3,5)_{\mathrm{d}}$ \\[0.6mm]
69 & $X_{14,16} \subset \mbP (1,1,5,7,8,9)$ & $4/45$ & $\frac{1}{5} (2,3)$, $\frac{1}{9} (1,8)_{\mathrm{d}}$ \\[0.6mm]
77 & $X_{16,18} \subset \mbP (1,1,6,8,9,10)$ & $1/15$ & $\frac{1}{2}$, $\frac{1}{3}$, $\frac{1}{10} (1,9)_{\mathrm{d}}$ \\[0.6mm]
81 & $X_{18,20} \subset \mbP (1,5,6,7,9,11)$ & $4/231$ & $\frac{1}{3}$, $\frac{1}{7} (2,5)_{\EI}$, $\frac{1}{11} (5,6)_{\mathrm{d}}$ \\[0.6mm]
82 & $X_{18,22} \subset \mbP (1,2,5,9,11,13)$ & $2/65$ & $\frac{1}{5} (1,4)$, $\frac{1}{13} (2,11)_{\mathrm{d}}$ \\[0.6mm]
\end{tabular}
\end{center}
\end{table}

\begin{table}[htb]
\begin{center}
\caption{Families indexed by $\IFii$}
\label{table:codim2Fanos-ii}
\begin{tabular}{clcc}
No. & $X_{d_1,d_2} \subset \mbP (a_0,\dots,a_5)$ & $(A^3)$ & Basket of singularities \\[0.5mm]
\hline \\[-3.5mm]
40 & $X_{10,12} \subset \mbP (1,1,3,4,5,9)$ & $2/9$ & $\frac{1}{3}_{\QI}$, $\frac{1}{9} (4,5)_{\mathrm{d}}$ \\[0.9mm]
43 & $X_{10,12} \subset \mbP (1,2,3,4,5,8)$ & $1/8$ & $3 \times \frac{1}{2}$, $\frac{1}{4}_{\QI}$, $\frac{1}{8} (3,5)_{\mathrm{d}}$ \\[0.9mm]
50 & $X_{10,14} \subset \mbP (1,2,3,5,7,7)$ & $2/21$ & $\frac{1}{3}$, $2 \times \frac{1}{7} (2,5)_{\mathrm{d}}$ \\[0.9mm]
52 & $X_{10,15} \subset \mbP (1,2,3,5,7,8)$ & $5/56$ & $\frac{1}{2}$, $\frac{1}{7} (2,5)_{\mathrm{d}}$, $\frac{1}{8} (3,5)_{\mathrm{d}}$ \\[0.6mm]
53 & $X_{12,13} \subset \mbP (1,3,4,5,6,7)$ & $13/210$ & $\frac{1}{2}$, $2 \times \frac{1}{3}$, $\frac{1}{5} (1,4)_{\EI}$, $\frac{1}{7} (3,4)_{\mathrm{d}}$ \\[0.9mm]
54 & $X_{12,14} \subset \mbP (1,1,3,4,7,11)$ & $2/11$ & $\frac{1}{11} (4,7)_{\mathrm{d}}$ \\[0.9mm]
56 & $X_{12,14} \subset \mbP (1,2,3,4,7,10)$ & $1/10$ & $4 \times \frac{1}{2}$, $\frac{1}{10} (3,7)_{\mathrm{d}}$ \\[0.9mm]
57 & $X_{12,14} \subset \mbP (1,2,3,5,7,9)$ & $4/45$ & $\frac{1}{3}$, $\frac{1}{5} (2,3)_{\mathrm{d}}$, $\frac{1}{9} (2,7)_{\mathrm{d}}$ \\[0.9mm]
58 & $X_{12,14} \subset \mbP (1,3,4,5,7,7)$ & $2/35$ & $\frac{1}{5}_{\QI}$, $2 \times \frac{1}{7} (3,4)_{\mathrm{d}}$ \\[0.9mm]
61 & $X_{12,15} \subset \mbP (1,1,4,5,6,11)$ & $3/22$ & $\frac{1}{2}$, $\frac{1}{11} (5,6)_{\mathrm{d}}$ \\[0.9mm]
62 & $X_{12,15} \subset \mbP (1,3,4,5,6,9)$ & $1/18$ & $\frac{1}{2}$, $3 \times \frac{1}{3}$, $\frac{1}{9} (4,5)_{\mathrm{d}}$ \\[0.9mm]
63 & $X_{12,15} \subset \mbP (1,3,4,5,7,8)$ & $3/56$ & $\frac{1}{4}$, $\frac{1}{7} (3,4)_{\mathrm{d}}$, $\frac{1}{8} (3,5)_{\mathrm{d}}$ \\[0.9mm]
65 & $X_{14,15} \subset \mbP (1,2,3,5,7,12)$ & $1/12$ & $\frac{1}{2}$, $\frac{1}{3}$, $\frac{1}{12} (5,7)_{\mathrm{d}}$ \\[0.9mm]
67 & $X_{14,15} \subset \mbP (1,3,4,5,7,10)$ & $1/20$ & $\frac{1}{2}$, $\frac{1}{4}$, $\frac{1}{5} (2,3)_{\QI}$, $\frac{1}{10} (3,7)_{\mathrm{d}}$ \\[0.9mm]
70 & $X_{14,16} \subset \mbP (1,3,4,5,7,11)$ & $8/165$ & $\frac{1}{3}$, $\frac{1}{5} (2,3)_{\QI}$, $\frac{1}{11} (4,7)_{\mathrm{d}}$ \\[0.9mm]
72 & $X_{15,16} \subset \mbP (1,2,3,5,8,13)$ & $1/13$ & $2 \times \frac{1}{2}$, $\frac{1}{13} (5,8)_{\mathrm{d}}$ \\[0.9mm]
73 & $X_{15,16} \subset \mbP (1,3,4,5,8,11)$ & $1/22$ & $2 \times \frac{1}{4}$, $\frac{1}{11} (3,8)_{\mathrm{d}}$ \\[0.9mm]
74 & $X_{14,18} \subset \mbP (1,2,3,7,9,11)$ & $2/33$ & $2 \times \frac{1}{3}$, $\frac{1}{11} (2,9)_{\mathrm{d}}$ \\[0.9mm]
79 & $X_{18,20} \subset \mbP (1,4,5,6,9,14)$ & $1/42$ & $2 \times \frac{1}{2}$, $2 \times \frac{1}{3}$, $\frac{1}{14} (5,9)_{\mathrm{d}}$ \\[0.9mm]
80 & $X_{18,20} \subset \mbP (1,4,5,7,9,13)$ & $2/91$ & $\frac{1}{7} (2,5)_{\QI}$, $\frac{1}{13} (4,9)_{\mathrm{d}}$ \\[0.9mm]
83 & $X_{20,21} \subset \mbP (1,3,4,7,10,17)$ & $1/34$ & $\frac{1}{2}$, $\frac{1}{17} (7,10)_{\mathrm{d}}$
\end{tabular}
\end{center}
\end{table}

%%%%%%%%%%%%%%%%%%%%%%%%%%%%%%%%%%
\subsection{Notation on weighted complete intersections}
%%%%%%%%%%%%%%%%%%%%%%%%%%%%%%%%%%

For a weighted projective space
\[
\mbP (a_0,\dots,a_n) = \Proj \mbC [x_0,\dots,x_n], \quad (\text{with $\deg (x_i) = a_i$}),
\]
and (weighted) homogeneous polynomials $f_j (x_0,\dots,x_n)$, $1 \le j \le m$, we denote by
\[
(f_1 = \cdots = f_m = 0) \subset \mbP (a_0,\dots,a_n)
\]
the closed subscheme defined by the homogeneous ideal $(f_1, \dots,f_m)$.

Let $X = X_{d_1, d_2} \subset \mbP (a_0,\dots,a_5)$ be a member of family No.~$\msi$ with $\msi \in \IF$.
Then $d_1 \ne d_2$ and $a_i = 1$ for some $0 \le i \le 5$ (see \cite[Section 9]{Okada1}).
Throughout the paper we assume that $d_1 < d_2$ and $1 = a_0 \le \cdots \le a_5$ unless otherwise specified.
The homogeneous coordinates of the ambient weighted projective space $\mbP (a_0,\dots,a_5)$ are usually denoted by $x, y, z, s, t, u$ and we have 
\[
\deg (x) = 1, \deg (y) = a_1, \deg (z) = a_2, \deg (s) = a_3, \deg (t) = a_4, \deg (u) = a_5.
\] 
We denote by 
\[
F_1 = F_1 (x,y,z,s,t,u), \quad F_2 = F_2 (x,y,z,s,t,u)
\] 
homogeneous polynomials of degree $d_1, d_2$ respectively which define $X$, that is, 
\[
X = (F_1 = F_2 = 0) \subset \mbP (a_0,\dots,a_5).
\]
We set $A = -K_X$.
Then $\mcO_X (A) \cong \mcO_X (1)$ and the Weil divisor class group $\Cl (X)$ is isomorphic to $\mbZ$ with positive generator $A$.

Let $v \in \{x, y, z, s, t, u\}$ be a homogeneous coordinate.
We define
\[
H_v := (v = 0) \cap X \subset X.
\]
We denote by $\msp_v$ the coordinate point of $\mbP (a_0,\dots,a_5)$ at which only the coordinate $v$ does not vanish.
For example, $\msp_u = (0\!:\!\cdots\!:\!0\!:\!1)$.
The restriction to $X$ of the projection from $\msp_v$ is denoted by $\pi_v$ and is also called the {\it projection from $\msp_v$}.
For example, 
\[
\pi_u \colon X \ratmap \mbP (a_0,\dots,a_4), \quad
(x\!:\!y\!:\!z\!:\!s\!:\!t\!:\!u) \mapsto (x\!:\!y\!:\!z\!:\!s\!:\!t).
\]
The projection $\pi_v$ is a generically finite dominant rational map which is defined possibly outside $\msp_v$, and the union of curves contracted by $\pi_v$ is denoted by $\Exc (\pi_v)$.

%%%%%%%%%%%%%%%%%%%%%%%%%%%%%%%%%%
\subsection{Distinguished singular points} \label{sec:distsing}
%%%%%%%%%%%%%%%%%%%%%%%%%%%%%%%%%%

Let $X$ be a member of family No.~$\msi$ with $\msi \in \IFi \cup \IFii$ (or more generally $\msi \in \IF$).
It is proved in \cite[Section 4.2]{Okada1} that there exists a Fano 3-fold $X'$ of Picard number $1$ which is birational to $X$ but not isomorphic to $X$, and $X$ admits a Sarkisov link $\sigma \colon X \ratmap X'$ which sits in the commutative diagram:
\[
\xymatrix{
Y \ar[d]_{\varphi} \ar@{-->}[r]^{\tau} & Y' \ar[d]^{\varphi'} \\
X \ar@{-->}[r]_{\sigma} & X'}
\]
Here $\varphi, \varphi'$ are divisorial contractions and $\tau$ is a flop.
Moreover $\varphi$ is the Kawamata blowup at a singular point $\msp \in X$.
We call $\msp$ a {\it distinguished singular point} of $X$ (see Section \ref{sec:flop} for further details).
For most of the families, $X$ admits a unique distinguished singular point.
However, for some families, $X$ admits several distinguished singular points, that is, $X$ admits several Sarkisov links to $X'$ (Note that the targets $X'$ are the same). 
In other words, a singular point $\msp \in X$ is distinguished if and only if the Kawamata blowup at $\msp$ leads to a Sarkisov link to $X'$.  
In Tables \ref{table:codim2Fanos-i} and \ref{table:codim2Fanos-ii}, distinguished singular points of $X$ are singular points given in 4th column marked $\mathrm{d}$ as a subscript.

\begin{Def}
Let $X$ be a member of family No.~$\msi$ with $\msi \in \IFi \cup \IFii$.
For a distinguished singular point $\msp \in X$, we denote by $\Upsilon_{\msp} \subset X$ the image via $\varphi$ of the union of the flopping curves on $Y$.
We then define
\[
\Upsilon_X := \bigcup_{\msp} \Upsilon_{\msp},
\]
where $\msp$ runs over the distinguished singular points on $X$.
\end{Def}

\begin{Rem}
After a suitable coordinate change if necessary, we may assume that a distinguished singular point $\msp$ is a coordinate point, i.e.\ typically $\msp = \msp_u$ and sometimes $\msp = \msp_t$.
If $\msp = \msp_v$, where $v \in \{t, u\}$, then $\Upsilon_{\msp} = \Exc (\pi_v)$ (see \cite[Section 4.2]{Okada1} for details).
\end{Rem}

\begin{Rem} \label{rem:numcharact}
We explain a numerical characterization of $\IFi$ and $\IFii$.
Let $X = X_{d_1,d_2} \subset \mbP (a_0,\dots,a_5)$ be a member of No.~$\msi$ with $\msi \in \IF$.
We assume that $d_1 \le d_2$ and $a_0 \le a_1 \le \cdots \le a_5$.
\begin{itemize}
\item $\msi \in \IFi$ if and only if $a_1 a_5 (-K_X)^3 \le 1$.
\item $\msi \in \IFii$ if and only if $1 < a_1 a_5 (-K_X)^3 \le 2$, $a_1 a_3 (-K_X)^3 \le 1$ and $\mathrm{wp} (\msp) > d_1$ for any distinguished singular point of $X$.
\end{itemize}
\end{Rem}

%%%%%%%%%%%%%%%%%%%%%%%%%%%%%%%%%%
\subsection{Generality assumptions and statements of main results}
%%%%%%%%%%%%%%%%%%%%%%%%%%%%%%%%%%

For a member $X$ of family No.~$\msi \in \IFi \cup \IFii$, we introduce the following conditions.

\begin{Cond} \label{cd}
\begin{enumerate}
\item $X$ is quasi-smooth.
\item The conditions given in \cite[Condition 3.1]{Okada1} are satisfied.
%\item The anticanonical linear system $\left|-K_X\right|$ contains a quasi-smooth member.
\end{enumerate}
\end{Cond}

It is clear that the above conditions are satisfied for general members of family No.~$\msi$.
We introduce further conditions on specific families.

\begin{Cond} \label{cdsp}
\begin{enumerate}
\item If $\msi \in \{40,43,50,52,53,67\} \subset \IFii$, then the $1$-dimensional scheme 
\[
L_{xy} := (x = y = 0) \cap X
\] 
is irreducible and reduced (cf.\ Lemma \ref{lem:irredLxy}). 
\item If $a_1 > 1$ (resp.\ $a_1 = 1$), then the conclusion of Lemma \ref{lem:multtang} (resp.\ Lemma \ref{lem:lcttang}) holds.
\item If $\msi \in \IFii$, then, for each distinguished singular point $\msp \in X$, the flopping locus $\Upsilon_{\msp} \subset X$ consists of $d_1 d_2/\mathrm{wp} (\msp)$ distinct irreducible and reducible curves (cf.\ Section \ref{sec:flop}).
\end{enumerate}
\end{Cond}

We can now state main results of this article in a precise form.

\begin{Thm}[$=$ Theorem \ref{mainthm1}] \label{thm:main1}
Let $X$ be a member of family No.~$\msi$ with
\[
\msi \in \IFi \cup \{79, 80, 83\} = \{42, 55, 66, 68, 69, 77, 79, 81, 80, 82, 83\} \subset \IFi \cup \IFii,
\]
satisfying \emph{Conditions \ref{cd}} and \emph{\ref{cdsp}}.
Then $\lct (X) = 1$.
In particular $X$ is K-stable and admits an orbifold K\"{a}hler--Einstein metric. 
\end{Thm}

\begin{Thm} \label{thm:main2}
Let $X$ be a member of family No.~$\msi$ with 
\[
\msi \in \IFii \setminus \{79, 80, 83\}
\] 
satisfying \emph{Conditions \ref{cd}} and \emph{\ref{cdsp}}, and let $X^{\circ} \subset X$ be the complement of the set of distinguished singular points on $X$.
Then $\lct_{X^{\circ}} (X) = 1$, or equivalently, for any effective $\mbQ$-divisor $D \sim_{\mbQ} - K_X$, the pair $(X, D)$ is log canonical along $X^{\circ}$.
\end{Thm}

Theorem \ref{thm:main2} is not enough to conclude the existence of K\"ahler--Einstein metric or K-stability.
However, this is clearly an important step toward the computation the global log canonical thresholds, which will be continued in a forthcoming paper.

%The first part of \ref{cdsp}(2) is clearly satisfied for a general members and, for the other conditions, detailed arguments will be given in the corresponding lemma.

%%%%%%%%%%%%%%%%%%%%%%%%%%%%%%%%%%
%%%%%%%%%%%%%%%%%%%%%%%%%%%%%%%%%%
\section{General computations} \label{sec:comp}
%%%%%%%%%%%%%%%%%%%%%%%%%%%%%%%%%%
%%%%%%%%%%%%%%%%%%%%%%%%%%%%%%%%%%

%%%%%%%%%%%%%%%%%%%%%%%%%%%%%%%%%%
\subsection{Multiplicities of some divisors} \label{sec:tangdiv}
%%%%%%%%%%%%%%%%%%%%%%%%%%%%%%%%%%

Let $X = X_{d_1,d_2} \subset \mbP := \mbP (1,a_1,\dots,a_5)$ be a member of family No.~$\msi$ with $\msi \in \IF$.
We assume that $1 \le a_1 \le a_2 \le \cdots \le a_5$.
Note that $a_1 < a_2$.

Suppose that $1 < a_1$ and let $\msp \in X \setminus H_x$ be a nonsingular point.
Then the linear system $|\mcI_{\msp} (a_1 A)|$ consists of a unique member and we denote it by $T_{\msp}$.
%By replacing coordinates, we may assume $\msp = \msp_x$ and in this case we have $T_{\msp} = H_y$.

\begin{Lem} \label{lem:multtang}
Suppose that $1 < a_1$ and that $X$ is general.
Then, for any point $\msp \in X \setminus H_x$, we have $\mult_{\msp} (T_{\msp}) \le 2$.
\end{Lem}

\begin{proof}
We compute the number of conditions imposed in order for $X$ to contain a point $\msp \in X \setminus H_x$ such that $\mult_{\msp} (T_{\msp}) > 2$.
By replacing coordinates, we may assume $\msp = \msp_x$.
In this case $T_{\msp} = H_y$ and we can write
\[
\begin{split}
F_1 &= \lambda_1 x^{d_1} + \alpha_1 y x^{d_1-a_1} + \beta_1 z x^{d_1-a_2} + \gamma_1 s x^{d_1-a_3} + \delta_1 t x^{d_1-a_4} + \varepsilon_1 u x^{d_1-a_5}, \\
F_2 &= \lambda_2 x^{d_2} + \alpha_2 y x^{d_2-a_1} + \beta_2 z x^{d_2-a_2} + \gamma_2 s x^{d_2-a_3} + \delta_2 t x^{d_2-a_4} + \varepsilon_2 u x^{d_2-a_5},
\end{split}
\]
for some $\lambda_j, \alpha_j,\dots,\varepsilon_j \in \mbC$.
We see that $\mult_{\msp} (H_y) \ge 2$ if and only if $\lambda_1 = \lambda_2 = 0$ and the rank of the matrix
\[
\begin{pmatrix}
\beta_1 & \gamma_1 & \delta_1 & \varepsilon_1 \\
\beta_2 & \gamma_2 & \delta_2 & \varepsilon_2
\end{pmatrix}
\]
if less than $2$, which imposes $5$ conditions for $\lambda_j, \alpha_j, \dots,\varepsilon_j$.
Although we do not make it explicit, it is clear that further additional conditions are imposed in order for the inequality $\mult_{\msp} (H_y) \ge 3$.

Let $\mcF$ be the space parametrizing the members of family No.~$\msi \in \IF$ and let $U = (x \ne 0) \subset \mbP$ be the open subset.
We define
\[
\mcW^{\circ} := \{\, (X, \msp) \mid \text{$\msp \in X$ and $\mult_{\msp} (T_{\msp}) \ge 3$} \,\} \subset \mcF \times U,
\]
and then define $\mcW$ to be the closure of $\mcW^{\circ}_k$ in $\mcF \times \mbP$.
Let $\mcW'$ be a component of $\mcW$ whose image on $\mbP$ intersects $U$. 
For a point $\msp \in U$, the fiber of the projection $\mcW' \to \mbP$ over $\msp$ is of codimension at least $6$.
Thus we have
\[
\dim \mcW' \le \dim \mbP + (\dim \mcF - 6) < \dim \mcF,
\]
and this proves the assertion.
\end{proof}

Suppose that $1 = a_1 < a_2$ and let $\msp \in X \setminus L_{xy}$ be a nonsingular point.
We also denote by $T_{\msp}$ the unique member of the linear system $|\mcI_{\msp} (A)|$.

\begin{Lem} \label{lem:lcttang}
Suppose that $a_1 = 1$ and that $X$ is general.
Then, for any point $\msp \in X \setminus L_{xy}$, we have $\lct (X, T_{\msp}) = 1$.
\end{Lem}

\begin{proof}
Let $\msp \in X \setminus L_{xy}$ be a point.
As in the proof of Lemma \ref{lem:multtang}, we see that $5$ conditions are imposed in order for $X$ to contain a point $\msp \in X \setminus L_{xy}$ such that $\mult_{\msp} (T_{\msp}) > 1$.
Thus, we cannot conclude that $\mult_{\msp} (T_{\msp}) = 1$ for all $\msp \in X \setminus L_{xy}$.
However, by the above dimension count, we can assume that a general $X$ does not contain a point $\msp \notin L_{xy}$ satisfying $\mult_{\msp} (T_{\msp}) > 1$ and some further additional conditions.
For example, we can and do assume that $(\prt F_1/\prt u) (\msp) \ne 0$ for any $\msp \notin L_{xy}$ such that $\mult_{\msp} (T_{\msp}) > 1$.

Now let $\msp \notin L_{xy}$ be a point.
By choosing coordinates, we assume $\msp = \msp_x$.
In this case we have $T_{\msp} = H_y$.
If $\mult_{\msp} (T_{\msp}) = 1$, then $(X, T_{\msp})$ is log canonical at $\msp$.
Hence we assume $\mult_{\msp} (T_{\msp}) > 1$.
%In this case we have $(\prt F_2/\prt u)(\msp) \ne 0$ by the above argument, which implies $a (\msp) \ne 0$.
By Tables \ref{table:codim2Fanos-i} and \ref{table:codim2Fanos-ii}, we have $2 a_1 > d_1$ and $2 a_1 \ge d_2$, which implies that $F_1$ is linear and $F_2$ is at most quadratic with respect to $u$, and we can write
\[
F_1 = u a - f, \quad
F_2 = x^{d_2-1} y + \alpha u^2 + u b + g,
\]
where $\alpha \in \{0,1\}$ and $a, b, f, g \in \mbC [x,y,z,s,t]$ with $a (\msp) \ne 0$, $f, g \in (y,z,s,t)^2$ and $b \in (y,z,s,t)$.
Set $\bar{a} = a (1,y,z,s,t)$ and similarly for $\bar{b}, \bar{f}, \bar{g}$.
By setting $x = 1$, passing to the completion of $\mcO_{X, \msp}$ and eliminating $u = \bar{a}^{-1} \bar{f}$, the germ $(X,\msp)$ is isomorphic to the origin of the hypersurface defined in $\mbA^4_{y,z,s,t}$ by the equation
\[
y + \alpha \tilde{a}^2 \tilde{f}^2 + \tilde{a}^{-1} \tilde{b} \tilde{f} + \tilde{g} = 0,
\]  
and $T_{\msp} = H_y$ corresponds to $y = 0$.
By Tables \ref{table:codim2Fanos-i} and \ref{table:codim2Fanos-ii}, we have $2 a_4 \le d_2$, which means that the quadratic part of $\tilde{g}$ is a general quadric in variables $y,z,s,t$.
By filtering off the terms divisible by $y$ we have
\[
(-1 + \cdots) y = \alpha \bar{a}^2 \bar{f}^2 + \bar{a}^{-1} \bar{b} \bar{f} + \bar{g},
\]
where $\bar{a} = \tilde{a} (1,0,z,s,t)$ and similarly for $\bar{b}, \bar{f}, \bar{g}$.
The quadratic part of the right-hand side, denoted by $q$, of the above equation is that of $\bar{g}$, which is a general quadric in variables $z,s,t$.
Thus the projectivised tangent cone $(q = 0) \subset \mbP^2_{z,s,t}$ is nonsingular.
Therefore, by \cite[Lemma 8.10]{Kol}, we have $\lct (X, T_{\msp}) = 1$.
\end{proof}

%%%%%%%%%%%%%%%%%%%%%%%%%%%%%%%%%%
\subsection{LCT along the flopping curves} \label{sec:flop}
%%%%%%%%%%%%%%%%%%%%%%%%%%%%%%%%%%

Let 
\[
X = X_{d_1, d_2} \subset \mbP := \mbP (1,a_1,\dots,a_5)
\] 
be a member of family No.~$\msi$ with $\msi \in \IF$ and $\msp \in X$ a distinguished singular point.

\begin{Lem} \label{lem:coordflopcurve}
We can choose homogeneous coordinates $x, x_{i_1}, x_{i_2}, x_{j_1}, x_{j_2}, x_k$ of $\mbP$, where $\deg x = 1, \deg x_{i_l} = a_{i_l}, \deg x_{j_l} = a_{j_l}, \deg x_k = a_k$ with $\{i_1,i_2,j_1,j_2,k\} = \{1,2,3,4,5\}$, such that  the following assertions hold.
\begin{enumerate}
\item Defining polynomials $F_1, F_2$ of $X$ with $\deg F_1 = d_1 < \deg F_2 = d_2$ can be written as
\[
F_1 = x_k x_{j_1} + G_1 (x,x_{i_1},x_{i_2},x_{j_1},x_{j_2}), \quad
F_2 = x_k x_{j_2} - G_2 (x,x_{i_1},x_{i_2},x_{j_1}),
\]
for some homogeneous polynomials $G_1, G_2$.
\item $a_{i_1} < a_{i_2} < a_k$.
\item $\msp = \msp_{x_k} \in X$ is of type $\frac{1}{a_k} (1,a_{i_1},a_{i_2})$ and we have $\wprod (\msp) = a_{i_1} a_{i_2}$.
\end{enumerate}
\end{Lem}

\begin{proof}
(1) is proved in \cite[Lemma 4.1]{Okada1}.
For most of the distinguished singular points, the weight $a_k$ is the maximum among $a_0,\dots,a_5$ so that (2) obviously follows.
For the remaining case, (2) follows from the table in \cite[Section 9]{Okada1}.
(3) follows immediately from (1) and (2).
\end{proof}

%\begin{Rem} \label{rem:twp}
%Under the choice of the coordinates as in Lemma \ref{lem:coordflopcurve}, we set 
%\[
%\mathrm{twp} (\msp) = a_{j_1} a_{j_2}
%\] 
%and call it the {\it tangent weight product} of the distinguished singular point $\msp$.
%\end{Rem}

Under the above choice of coordinates, we say that 
\[
T_1 := H_{x_{j_1}} \quad \text{and} \quad T_2 := H_{x_{j_2}}
\] 
are the {\it first} and {\it second tangent divisors} of $X$ at $\msp$, respectively.
Note that 
\[
\Upsilon_{\msp} = \Exc (\pi_{x_k}) = T_1 \cap T_2.
\]

\begin{Lem} \label{lem:normalT1}
The divisor $T_1$ has only isolated singularities along $T_1 \setminus \Upsilon_{\msp}$.
\end{Lem}

\begin{proof}
We choose coordinates as in Lemma \ref{lem:coordflopcurve}.
We set
\[
\Xi := \left( \frac{\prt G_1}{\prt x} = \frac{\prt G_1}{\prt x_{i_1}} = \frac{\prt G_1}{\prt x_{i_2}} = \frac{\prt G_1}{\prt x_{j_2}} = 0 \right) \cap T_1,
\]
and claim that $\Xi$ is a finite set of points.
Indeed, if $\Xi$ contains a curve, then
\[
\Xi' := \Xi \cap \left( x_k + \frac{\prt G_1}{\prt x_{j_1}} = 0 \right) \ne \emptyset.
\]
By considering the Jacobi matrix of the affine cone of $X$, we see that $X$ is not quasi-smooth at each point of $\Xi'$.
This is a contradiction and $\Xi$ is a finite set of points.

Let $J_{T_1}$ be the Jacobi matrix of the affine cone of $T_1$.
Since $T_1 = (x_{j_1} = G_1 = x_k x_{j_2} - G_2 = 0)$, we have
\[
J_{T_1} =
\begin{pmatrix}
0 & 0 & 0 & 1 & 0 & 0 \\[1mm]
\frac{\prt G_1}{\prt x} & \frac{\prt G_1}{\prt x_{i_1}} & \frac{\prt G_1}{\prt x_{i_2}} & \frac{\prt G_1}{\prt x_{j_1}} & \frac{\prt G_1}{\prt x_{j_2}} & 0 \\[2mm]
- \frac{\prt G_2}{\prt x} & - \frac{\prt G_2}{\prt x_{i_1}} & - \frac{\prt G_2}{\prt x_{i_2}} & - \frac{\prt G_2}{\prt x_{j_1}} & x_k - \frac{\prt G_2}{\prt x_{j_2}} & x_{j_2}
\end{pmatrix}.
\]
On the open set $T_1 \setminus \Upsilon_{\msp}$, the section $x_{j_2}$ does not vanish.
Thus $J_{T_1} (\msq)$ is of rank $3$ for any point $\msq \in T_1 \setminus (\Upsilon_{\msp} \cup \Xi)$.
This completes the proof.
\end{proof}

We set
\[
\overline{\Upsilon}_{\msp} := (x_{j_1} = x_{j_2} = G_1 = G_2 = 0) \subset \overline{\mbP} := \mbP (1,a_{i_1}, a_{i_2}, a_{j_1}, a_{j_2}),
\]
which is the image of $\Upsilon_{\msp}$ via the projection $\pi_{x_k}$.
In the following we assume that $\overline{\Upsilon}_{\msp}$ consists of $e := (d_1 d_2)/(a_{i_1} a_{i_2})$ distinct nonsingular points of $\overline{\mbP}$, which is the case when $X$ is general, and this assumption is equivalent to Condition \ref{cdsp}(3).
Then we have
\[
T_1 \cdot T_2 = \Gamma_1 + \cdots + \Gamma_e,
\]
where 
\begin{itemize}
\item each $\Gamma_l$ is a nonsingular rational curve such that $(A \cdot \Gamma_l) = 1/a_k$,
\item $\Gamma_l \cap \Gamma_m = \{\msp\}$ for $l \ne m$, and
\item $\Upsilon_{\msp} = \Exc (\pi_{x_k}) = \cup_{l=1}^e \Gamma_l$.
\end{itemize}

\begin{Lem}
Let $\msp \in X$ be a distinguished singular point of a member $X$ of family No.~$\msi$ with $\msi \in \IF$.
Assume that $\overline{\Upsilon}_{\msp}$ consists of $e := (d_1 d_2)/(a_{i_1} a_{i_2})$ distinct nonsingular points of $\overline{\mbP}$.
Then the following assertions hold.
\begin{enumerate}
\item The surface $T_1$ is normal, and it is nonsingular along $\Upsilon_{\msp} \setminus \{\msp\}$.
\item For distinct $l, m \in \{1,\dots,e\}$, we have
\[
(\Gamma_l \cdot \Gamma_m)_{T_1} = \frac{a_{i_1} a_{i_2}}{d_1 a_k}.
\]
\item For $l \in \{1,\dots,e\}$, we have
\[
(\Gamma_l^2)_{T_1} = - 1 + \frac{a_{i_1} a_{i_2}}{d_1 a_k}.
\]
\end{enumerate}
\end{Lem}

\begin{proof}
We choose coordinates as in Lemma \ref{lem:coordflopcurve}.
The assumption on $\overline{\Upsilon}_{\msp}$ implies the the following Jacobian matrix of $\overline{\Upsilon}_{\msp} \subset \overline{\mbP}$, 
\[
\begin{pmatrix}
\frac{\prt \bar{G}_1}{\prt x} & \frac{\prt \bar{G}_1}{\prt x_{j_1}} & \frac{\prt \bar{G}_1}{\prt x_{j_2}} \\[2mm]
\frac{\prt \bar{G}_2}{\prt x} & \frac{\prt \bar{G}_2}{\prt x_{j_1}} & \frac{\prt \bar{G}_2}{\prt x_{j_2}} 
\end{pmatrix}
\]
is of rank $2$ at any point of $\overline{\Upsilon}_{\msp} \subset \mbP (1,a_{i_1}, a_{i_2})$, where
\[
\bar{G}_1 = G_1 (x, x_{i_1}, x_{i_2}, 0, 0), \quad \text{and} \quad
\bar{G}_2 = G_2 (x,x_{i_1}, x_{i_2}, 0).
\]
From this we deduce that the surface $T_1$ is nonsingular along $\Upsilon_{\msp} \setminus \{\msp\}$.
Combining this with Lemma \ref{lem:normalT1}, we conclude that $T_1$ has only isolated singular points and thus it is normal.

Let $\varphi \colon Y \to X$ be the Kawamata blowup of $X$ at $\msp$ with exceptional divisor $E$ and let $\psi \colon \tilde{T}_1 \to T_1$ be the restriction of $\varphi$ to the proper transform $\tilde{T}_1$ of $T_1$.
We set $F = E|_{\tilde{T}_1}$.
Note that $\varphi^* T_1 = \tilde{T}_1 + \frac{d_1}{a_k} E$.
Hence we have
\[
(F^2)_{\tilde{H}_1} 
= (E^2 \cdot \varphi^*T_1 - \tfrac{d_1}{a_k} E)_Y 
= - \frac{d_1 a_k}{a_{i_1} a_{i_2}}.
\]

For $l = 1,\dots,e$, we write $\varphi^* \Gamma_l \equiv \tilde{\Gamma}_l + c_l F$
and compute the number $c_l$ by taking intersection with $F$.
Since $\tilde{\Gamma}_l$ intersects $F$ transversally at one point which is a nonsingular point of both $\tilde{\Gamma}_l$ and $F$, we have
\[
0 = (\varphi^*\Gamma_l \cdot F)_{\tilde{T}_i} 
= (\tilde{\Gamma}_l \cdot F)_{\tilde{T}_1} + c_l (F^2)_{\tilde{T}_1} 
= 1 - \frac{d_1 a_k}{a_{i_1} a_{i_2}}c_l,
\]
which shows
\[
\varphi^* \Gamma_l = \tilde{\Gamma}_l + \frac{a_{i_1} a_{i_2}}{d_1 a_k} F.
\]
For $l \ne m$, we have $\tilde{\Gamma}_l \cap \tilde{\Gamma}_m = \emptyset$ and hence
\[
(\Gamma_l \cdot \Gamma_m)_{T_1} 
= (\psi^*\Gamma_l \cdot \tilde{\Gamma}_m)_{\tilde{T}_1} 
= \frac{a_{i_1} a_{i_2}}{d_1 a_k},
\]
and (2) is proved.

We have
\[
\frac{a_{j_2}}{a_k} 
= (T_2|_{T_1} \cdot \Gamma_l)_{H_1} 
= (\Gamma_l)^2_{H_1} + \sum_{m \ne l} (\Gamma_l \cdot \Gamma_m)_{H_1}
= (\Gamma^2_l)_{H_1}
 + (e-1) \frac{a_{i_1} a_{i_2}}{d_1 a_k}.
 \]
This, together with $e = (d_1 d_2)/(a_{i_1} a_{i_2})$ and $d_2 = a_k + a_{j_2}$, shows (3).
\end{proof}

\begin{Prop} \label{prop:lctfflopcurve}
Let $\msp \in X$ be a distinguished singular point of a member $X$ of family No.~$\msi \in \IF$.
Assume that $\wprod (\msp) \ge d_1$ and that $\overline{\Upsilon}_{\msp}$ consists of $e := (d_1 d_2)/(a_{i_1} a_{i_2})$ distinct nonsingular points of $\overline{\mbP}$.
Then $\lct_{\msq} (X) \ge 1$ for any nonsingular point $\msq \in \Upsilon_{\msp}$ of $X$.
\end{Prop}

\begin{proof}
Let $\msq \in \Upsilon_{\msp} = \cup \Gamma_l$ be a nonsingular point of $X$.
There exists a unique curve $\Gamma_l$ passing through $\msq$.
Without loss of generality, we may assume $\msq \in \Gamma_1$.
The pair $(T_1, \frac{1}{a_{j_2}} T_2|_{T_1})$ is log canonical at $\msq$ since $T_1$ is nonsingular along $\Upsilon_{\msp} \setminus \{\msp\} = (\cup \Gamma_l) \setminus \{\msp\}$, $\Gamma_1$ is nonsingular and $\msq \notin \Gamma_l$ for $l \ne 1$, and hence $(X, \frac{1}{a_{j_2}} T_2)$ is log canonical at $\msq$.

Suppose that $\lct_{\msq} (X) < 1$.
Then there exists an irreducible $\mbQ$-divisor $D \sim_{\mbQ} A$ such that $(X, D)$ is not log canonical at $\msq$ and $\Supp (D) \ne T_2$.
%By Remark \ref{rem:compD} below, we may assume that $D$ does not contain $T_1$ in its support.
Then $D|_{T_1}$ is an effective $\mbQ$-divisor on $T_1$ and $(T_1, D|_{T_1})$ is not log canonical at $\msq$.
%Then we can write
%\[
%D|_{H_1} = \gamma \Gamma_1 + \Delta,
%\]
%where $\Delta$ is an effective $1$-cycle on $H_1$ such that $\Gamma_1 \not\subset \Supp (\Delta)$.
For $\varepsilon \in \mbQ$, we set 
\[
D_{\varepsilon} = (1 + \varepsilon) D - \frac{\varepsilon}{a_{j_2}} T_2 \sim_{\mbQ} A.
\]
We can find $\varepsilon \ge 0$ such that the $1$-cycle
\[
D_{\varepsilon} |_{T_1} = \varepsilon D|_{T_1} - \frac{\varepsilon}{a_{j_2}} \sum \Gamma_l.
\]
is effective and $\Gamma_l \not\subset \Supp (D_{\varepsilon}|_{T_1})$ for some $l \in \{1,\dots,e\}$.
Note that $(X, D_{\varepsilon}|_{T_1})$ is not log canonical at $\msq$ (see Remark \ref{rem:covex}).
If $\Gamma_1 \not\subset \Supp (D_{\varepsilon}|_{T_1})$, then
\[
\frac{1}{a_k} = (D|_{T_1} \cdot \Gamma_1)_{T_1} > 1.
\]
This is a contradiction.
Thus $\Gamma_l \not\subset \Supp (D_{\varepsilon}|_{T_1})$ for some $l \ne 1$.
After replacing $D$ with $D_{\varepsilon}$, we may assume that $(X, D|_{T_1})$ is not log canonical at $\msq$, $D|_{T_1}$ contains $\Gamma_1$ but does not contain $\Gamma_2$ in its support.

We write
\[
D|_{T_1} = \gamma \Gamma_1 + \Delta,
\]
where $\gamma > 0$ and $\Gamma_l \not\subset \Supp (\Delta)$ for $l = 1,2$.
By taking intersection with $\Gamma_2$, we have
\[
\frac{1}{a_k} = (D|_{T_1} \cdot \Gamma_2)_{T_1} 
\ge \gamma (\Gamma_1 \cdot \Gamma_2)_{T_1} 
= \gamma \frac{a_{i_1} a_{i_2}}{d_1 a_k},
\]
which implies
\begin{equation} \label{eq:exclflopcurves1}
\gamma \le \frac{d_1}{a_{i_1} a_{i_2}}=  \frac{d_1}{\wprod (\msp)}.
\end{equation}
Since $\wprod (\msp) \ge d_1$, we have $\gamma \le 1$.
Thus we can apply the inversion of adjunction formula and we have
\begin{equation} \label{eq:exclflopcurves2}
\begin{split}
\frac{1}{a_k} + \left(1 - \frac{a_{i_1} a_{i_2}}{d_1 a_k} \right) \gamma  
&= ((D|_{T_1} - \gamma \Gamma_1) \cdot \Gamma_1)_{T_1}  \\
&= (\Delta \cdot \Gamma_1)_{T_1} 
\ge \mult_{\msp} (\Delta|_{\Gamma_1}) 
> 1.
\end{split}
\end{equation}
Combining \eqref{eq:exclflopcurves1} and \eqref{eq:exclflopcurves2}, we obtain
\[
\frac{1}{a_k} + \left(1 - \frac{a_{i_1} a_{i_2}}{d_1 a_k} \right) \frac{d_1}{a_{i_1} a_{i_2}} > 1,
\]
which is equivalent to $d_1 > a_{i_1} a_{i_2} = \wprod (\msp)$.
This is a contradiction and the proof is completed.
\end{proof}

%\begin{Rem} \label{rem:compD}
%Let $V$ be a normal projective $\mbQ$-factorial variety, $\msp \in V$ a point and $D, S$ effective $\mbQ$-divisors such that $D \sim_{\mbQ} S \sim_{\mbQ} -K_V$.
%Suppose that $(V, D)$ is not log canonical at $\msp$ and $(V, S)$ is log canonical at $\msp$.
%For a non-negative $\varepsilon \in \mbQ$, we set $D_{\varepsilon} = (1 + \varepsilon) D - \varepsilon S \sim_{\mbQ} - K_V$.
%If $D_{\varepsilon}$ is effective, then $(V, D_{\varepsilon})$ is not log canonical at $\msp$.
%Indeed, we have
%\[
%D = \frac{1}{1+\varepsilon} D_{\varepsilon} + \frac{\varepsilon}{1 + \varepsilon} S,
%\]
%and thus log canonicity of $(X, D_{\varepsilon})$ implies that of $(X, D)$ (see Remark \ref{rem:covex}).
%\end{Rem}

%%%%%%%%%%%%%%%%%%%%%%%%%%%%%%%%%%%%
\subsection{LCT at some distinguished singular points}
%%%%%%%%%%%%%%%%%%%%%%%%%%%%%%%%%%%%

\begin{Lem} \label{lem:somedistsingpt}
Let $\msp \in X$ be a distinguished singular point of a member $X$ of family No.~$\msi$ with $\msi \in \IF$.
We choose homogeneous coordinates as in \emph{Lemma \ref{lem:coordflopcurve}}.
If $H_x \cap \Upsilon_{\msp} = \{\msp\}$ and $a_k a_{j_2} (A^3) \le 2$, then $\lct_{\msp} (X) \ge 1$.
\end{Lem}

\begin{proof}
Suppose that $\lct_{\msp} (X) < 1$.
Then there is an irreducible effective $\mbQ$-divisor $D \sim_{\mbQ} A$ such that $(X, D)$ is not log canonical at $\msp$.
In particular $(\check{X}_{\msp}, \check{D})$ is not log canonical at $\check{\msp}$, where $\rho_{\msp} \colon \check{X}_{\msp} \to X$ is the index one cover of $\msp \in X$ and $\check{D} = \rho_{\msp}^*D$.

Since $(X, H_x)$ is log canonical at $\msp$, we have $H_x \ne \Supp (D)$ and $D \cdot H_x$ is an effective $1$-cycle.
Set $b = \lcm \{a_{j_1}, a_{j_2}\}$, $b_1 = b/a_{j_1}$, $b_2 = b/a_{j_2}$ and let $T$ be the defined by $\alpha_1 x_{j_1}^{b_1} + \alpha_2 x_{j_2}^{b_2}$ for general $\alpha_1, \alpha_2$.
Since $H_x \cap (x_{j_1} = x_{j_2} = 0)$ consists of finite set of points, we may assume that $\Supp (T)$ does not contain any component of $D \cdot H_x$.
Since $\mult_{\check{\msp}} (\check{D}) > 1$ and $\mult_{\check{\msp}} (\check{T}) \ge 2 b_2$, where $\check{T} = \rho_{\msp}^* T$, we have
\[
a_k b (A^3) = a_k (D \cdot H_x \cdot T) \ge (\check{D} \cdot \check{H}_x \cdot \check{T})_{\check{\msp}} > 2 b_2 = 2 b/a_{j_2}.
\]
This is a contradiction.
Therefore $\lct_{\msp} (X) \ge 1$.
\end{proof}

%%%%%%%%%%%%%%%%%%%%%%%%%%%%%%%%%%
%%%%%%%%%%%%%%%%%%%%%%%%%%%%%%%%%%
\section{GLCT for families indexed by $\IFi \cup \IFii$} \label{sec:nonsingpts}
%%%%%%%%%%%%%%%%%%%%%%%%%%%%%%%%%%
%%%%%%%%%%%%%%%%%%%%%%%%%%%%%%%%%%

We compute GLCT of a member $X$ of family No.~$\msi$ with $\msi \in \IFi \cup \IFii$ and give proofs of Theorems \ref{thm:main1} and \ref{thm:main2} at the end of this section.

%%%%%%%%%%%%%%%%%%%%%%%%%%%%%%%%%%
\subsection{Nonsingular points contained in $L_{xy}$}
%%%%%%%%%%%%%%%%%%%%%%%%%%%%%%%%%%

For a member $X$ of family No.~$\msi \in \IFi \cup \IFii$, we have $a_1 < a_2$.
This implies that the $1$-dimensional scheme 
\[
L_{xy} := H_x \cap H_y
\] 
does not depend on the choice of homogeneous coordinates.
%We prove that $L_{xy}$ is an irreducible and reduced curve with $\Sing (L_{xy}) \subset \Sing (X)$.

\begin{Lem} \label{lem:irredLxy}
Let $X$ be a member of family No.~$\msi$ with $\msi \in \IFi \cup \IFii$.
For $\msi \in \{40,43,50,52,53,67\} \subset \IFii$, we assume that the condition given in the $4$th column of \emph{Tables \ref{table:Lxy-ii}} is satisfied.
Then $L_{xy}$ is an irreducible and reduced curve such that $\Sing (L_{xy}) \subset \Sing (X)$.
\end{Lem}

\begin{proof}
We write $F_1 = G_1 + H_1$ and $F_2 = G_2 + H_2$, where $G_j \in \mbC [z,s,t,u]$ and $H_j \in (x,y) \subset \mbC [x,y,\dots,u]$ for $j = 1,2$.
Then $L_{xy}$ is isomorphic to the closed subscheme defined by $G_1 = G_2 = 0$ in $\mbP (a_2,a_3,a_4,a_5)$.
Quasi-smoothness of $X$ implies the presence of some monomials in $G_1, G_2$ and after replacing coordinates, the equations $G_1 = G_2 = 0$ can be transformed into the form given in the second column of Tables \ref{table:Lxy-i} and \ref{table:Lxy-ii} (see Example \ref{ex:Lxyeq} below).

Once we know the equations, it is then straightforward to determine $\Sing (L_{xy})$ via the computation of the Jacobian matrix and the description of $\Sing (L_{xy})$ is given in the third column of each table.
It is also straightforward to check that $L_{xy}$ is irreducible (see Example \ref{ex:Lxyirred} below).
\end{proof}

\begin{table}[htb]
\begin{center}
\caption{$L_{xy}$ for type $\IFi$ families}
\label{table:Lxy-i}
\begin{tabular}{cccc}
No. & Equations $G_1 = G_2 = 0$ & $\Sing (L_{xy})$ & Irr \\
\hline \\[-3mm]
42 & $u z + t z + s^2 = u t + z^3 = 0$ & $\emptyset$ & \\
55 & $u z + s^2 = u s + t^2 + s z^2 = 0$ & $\emptyset$ & \\
66 & $u z + t^2 = u s + z^3 = 0$ & $\emptyset$ & \\
68 & $u s + t^2 = u t + z^3 = 0$ & $\emptyset$ & \\
69 & $u z + s^2 = u s + t^2 = 0$ & $\emptyset$ & \\
77 & $u z + s^2 = u s + t^2 + \lambda z^3 = 0$, $\lambda \ne 0$ & $\emptyset$ & \\
81 & $u s + t^2 + z^3 = u t + \lambda s^2 z = 0$, $\lambda \ne 0$ & $\emptyset$ & \\
82 & $u z + s^2 = u s + t^2 = 0$ & $\{\msp_z\}$ & $\pi_z$ \\
\end{tabular}
\end{center}
\end{table}

\begin{table}[htb]
\begin{center}
\caption{$L_{xy}$ for type $\IFii$ families}
\label{table:Lxy-ii}
\begin{tabular}{ccccc}
No. & Equations $G_1 = G_2 = 0$ & $\Sing (L_{xy})$ & Irr & Cond \\
\hline \\[-3mm]
40 & $t^2 + s z^2 = u z + s^3 = 0$ & $\{\msp_u\}$ & $\pi_u$ & $s z^2 \in F_1$  \\
43 & $t^2 + s z^2 = u s + z^4 = 0$ & $\{\msp_s, \msp_u\}$ & $\pi_u$ & $s z^2 \in F_1$ \\
50 & $u z + t z + s^2 = u t + s z^3 = 0$ & $\emptyset$ & & $s z^3 \in F_1$ \\
52 & $t z + s^2 = u t + z^5 = 0$ & $\{\msp_u\}$ & $\pi_u$ & $s^2 \in F_1$ \\
53 & $u s + t^2 + z^3 = u t + \lambda s z^2 = 0$, $\lambda \ne 0$ & $\{\msp_s\}$ & $s \ne 0$ & $s z^2 \in F_1$ \\
54 & $s^3 + z^4 = u z + t^2 = 0$ & $\{\msp_u\}$ & $\pi_u$ & \\
56 & $s^3 + z^4 = u s + t^2 = 0$ & $\{\msp_u\}$ & $\pi_u$ & \\
57 & $u z + t s = u s + t^2 + s z^3 = 0$ & $\emptyset$ & & $t s \in F_1$\\
58 & $u s + t s + z^3 = u t + s^2 z = 0$ & $\emptyset$ & & \\
61 & $t^2 + z^3 = u z + s^3 = 0$ & $\{\msp_u\}$ & $\pi_u$ & \\
62 & $t^2 + z^3 = u t + s^3 = 0$ & $\{\msp_u\}$ & $\pi_u$ & \\
63 & $u z + t s = u t + \lambda t z^2 + s^3 = 0$, $\lambda \ne 0$ & $\emptyset$ & & $\not\exists (1,3,5,8)$ \\
65 & $t^2 + s^3 z = u z + s^3 = 0$ & $\{\msp_u\}$ & $\pi_u$ & \\
67 & $u z + t^2 + s^2 z = u s + \lambda t z^2 = 0$, $\lambda \ne 0$ & $\emptyset$ & & $t z^2 \in F_2$ \\
70 & $t^2 + s^2 z = u s + z^4 = 0$ & $\{\msp_u\}$ & $\pi_u$ & \\
72 & $s^3 + z^5 = u z + t^2 = 0$ & $\{\msp_u\}$ & $\pi_u$ & \\
73 & $u z + s^3 = u s + t^2 + \lambda z^4 = 0$, $\lambda \ne 0$ & $\emptyset$ & & \\
74 & $u z + s^2 = u s + t^2 + \lambda z^6 = 0$, $\lambda \ne 0$ & $\emptyset$ & & \\
79 & $t^2 + s^3 = u s + z^4 = 0$ & $\{\msp_u\}$ & $\pi_u$ & \\
80 & $u z + t^2 = u s + z^4 = 0$ & $\{ \msp_s \}$ & $\pi_s$ & \\
83 & $t^2 + z^5 = u z + s^3 = 0$ & $\{ \msp_u\}$ & $\pi_u$ & \\
\end{tabular}
\end{center}
\end{table}

\begin{Rem}
We explain the conditions given in the 5th column of Table \ref{table:Lxy-ii}.

The condition of the form $M \in F_i$, where $M$ is a monomial, means that the monomial appears in $F_i$ with non-zero coefficient.
This is clearly a generality condition, i.e.\ this is satisfied for a general member.

We explain the condition $\not\exists (1,3,5,8)$ for a member $X$ of family No.~$63$.
This means that $X$ does not contain a weighted complete intersection curve (WCI curve, for short) of type $(1,3,5,8)$.
Here, a WCI curve of type $(1,3,5,8)$ can be explicitly given as $x = y = s = \lambda u + \mu z^2 = 0$, so that such curves form $1$-dimensional family.
On the other hand, $2$ conditions are imposed for $X$ to contain a given WCI curve of type $(1,3,5,8)$.
Thus this condition is also a generality condition.
\end{Rem}

\begin{Ex} \label{ex:Lxyeq}
We explain how to obtain equations $G_1 = G_2 = 0$ in Tables \ref{table:Lxy-i} and \ref{table:Lxy-ii}.
Let $X$ be a member of family No.~$\msi$.

Suppose that $\msi = 40$.
Then we can write
\[
G_1 = \alpha t^2 + \beta s z^2, \quad
G_2 = \gamma u z + \delta t s z + \varepsilon s^3 + \lambda z^4,
\]
for some $\alpha, \beta, \dots, \lambda \in \mbC$.
We have $\alpha \ne 0, \gamma \ne 0, \varepsilon \ne 0$ since $X$ is quasi-smooth.
Re-scaling $u$ and $s$, we may assume $\gamma = \varepsilon = 1$.
Then, by replacing $u$, we can eliminate monomials in $G_2$ divisible by $u$, so that we may assume $G_2 = u z + s^3$. 
By the condition $s z^2$, we have $\beta \ne 0$.
Then, re-scaling $t$, we may assume $G_1 = t^2 + s z^2$, and we obtain the equations in the table.

Suppose that $\msi = 63$.
Then we can write
\[
G_1 = \alpha u z + \beta t s + \gamma z^3, \quad
G_2 = \delta u t + \lambda t z^2 + \varepsilon s^3,
\]
for some $\alpha, \beta, \dots, \lambda \in \mbC$.
By quasi-smoothness of $X$, we have $\alpha \ne 0, \beta \ne 0, \delta \ne 0, \varepsilon \ne 0$.
Replacing $u$ (in order to eliminate $z^3$), we can assume $\gamma = 0$.
Then, re-scaling $z, s, u$, we may assume that $\alpha = \gamma = \delta = \varepsilon = 1$.
We have $\lambda \ne 0$ because otherwise $X$ contains the WCI curve $(x = y = s = u = 0)$ of type $(1, 3, 5, 8)$ which is impossible the condition $\not\exists (1, 3, 5, 8)$.
Thus we obtain the desired equations.
\end{Ex}

\begin{Ex} \label{ex:Lxyirred}
We explain how to check that $L_{xy}$ is irreducible.
Note that $L_{xy}$ is clearly reduced.

First we consider $L_{xy}$ which is nonsingular.
By a choice of coordinates, we may assume that $H_x$ is quasi-smooth by Condition \ref{cd}(1).
Then $H_x$ is a normal projective variety and $L_{xy}$ is defined by $y = 0$ on $H_x$.
It follows that $L_{xy}$ is the support of an ample divisor and thus it is connected.
Since $L_{xy}$ is nonsingular, it is irreducible.

Next we consider $L_{xy}$ which is singular.
Suppose that $\pi_v$ is given (for some $v \in \{z,s,u\}$) in the 4th column of the tables.
Then the projection $\pi_v|_{L_{xy}} \colon L_{xy} \to \pi_v (L_{xy})$ is a birational morphism onto its image $\pi_v (L_{xy})$ and we can check that $\pi_v (L_{xy})$ is an irreducible curve.
This shows that $L_{xy}$ is irreducible.
For example, we consider family No.~$40$.
Then 
\[
L_{xy} \cong (t^2 + s z^2 = u z + s^3 = 0) \subset \mbP (3,4,5,9)
\]
and
\[
\pi_u (L_{xy}) \cong (t^2 + s z^2 = 0) \subset \mbP (3,4,5)
\]
and $\pi_u$ is clearly irreducible.

We consider family No.~$53$.
Let $U \subset X$ be the open subset on which $s \ne 0$.
Then $L_{xy} \cap U$ is isomorphic to the $\mbZ/5 \mbZ$-quotient of the affine curve
\[
\begin{split}
& (u + t^2 + z^3 = u t + \lambda z^2 = 0) \subset \mbA^3_{z,t,u}, \\
\cong & \ ((t^2 + z^3) t - \lambda z^2 = 0) \subset \mbA^2_{z,t},
\end{split}
\]
which is irreducible since $\lambda \ne 0$.
Moreover $L_{xy} \cap (X \setminus U)$ consists of $2$ points.
This shows that $L_{xy}$ is irreducible.
\end{Ex}

\begin{Rem} \label{rem:multLxy}
Let $X$ be a member of family No.~$\msi \in \IFi \cup \IFii$ and $\msp$ a point of $L_{xy}$.
Let $\rho_{\msp} \colon \check{X}_{\msp} \to X$ be the index one cover of an open neighborhood of $\msp \in X$ and let $\check{\msp} \in \check{X}_{\msp}$ be the preimage of $\msp$.
When $\msp$ is a nonsingular point of $L_{xy}$ we have $\mult_{\check{\msp}} (\check{L}_{xy}) = \mult_{\msp} (L_{xy}) = 1$.
For a singular point $\msp$ of $L_{xy}$ described in the 3rd column of Tables \ref{table:Lxy-i} and \ref{table:Lxy-ii}, we have
\[
\mult_{\check{\msp}} (\check{L}_{xy}) =
\begin{cases}
4, & \text{for family No.~56 and $\msp = \msp_u$}, \\
3, & \text{for families No.~54, 62, 72 and $\msp = \msp_u$} \\ 
2, & \text{otherwise}.
\end{cases}
\]
This implies $\mult_{\check{\msp}} (\check{L}_{xy}) \le a_1$.

We explain how to compute $\mult_{\check{\msp}} (\check{L}_{xy})$ by an example.
We consider $\msp = \msp_z$ for family No.~82.
The curve $\check{L}_{xy}$ is isomorphic to the curve defined by
\[
u + s^2 = u s + t^2 = 0
\]
in the affine space $\mbA^3$ with coordinates $s, t, u$, and $\check{\msp}$ corresponds to the origin.
These equations are simply obtained by setting $z = 1$.
It is then easy to see that $\mult_{\check{\msp}} (\check{L}_{xy}) = 2$.
Computations for the other instances are the same.
\end{Rem}

\begin{Prop} \label{prop:compLxy}
Let $X$ be a member of family No.~$\msi$ with $\msi \in \IFi \cup \IFii$ and let $\msp \in L_{xy}$ be a nonsingular point of $X$.
Then $\lct_{\msp} (X) \ge 1$.
\end{Prop}

\begin{proof}
Let $S_1 \in |A|$ and $S_2 \in |a_1 A|$ be general members.
%Then $S_1$ is quasi-smooth by Condition \ref{cd}(1), hence $\mult_{\msp} (S_1) = 1$ and $(X, S_1)$ is log canonical at $\msp$.
%Let $S_2 \in |a_1 A|$ be a general member so that $L_{xy} = S_1 \cap S_2$.
We have $\mult_{\msp} (S_1) = 1$ since $L_{xy} = S_1 \cap S_2$ and $\mult_{\msp} (L_{xy}) = 1 \le a_1$ by Lemma \ref{lem:irredLxy}.
It follows that $(X, S_1)$ is log canonical at $\msp$.
Thus we can apply Lemma \ref{lem:exclL} and conclude that $\lct_{\msp} (X) \ge 1$.
\end{proof}

%%%%%%%%%%%%%%%%%%%%%%%%%%%%%%%%%%
\subsection{Nonsingular points not contained in $L_{xy}$}
%%%%%%%%%%%%%%%%%%%%%%%%%%%%%%%%%%

\subsubsection{Families No.~$\msi$ with $\msi \in \IFi$}

\begin{Lem} \label{lem:isol-i}
Let $X$ be a member of family No.~$\msi$ with $\msi \in \IFi$ and $\msp$ a nonsingular point of $X$.
\begin{enumerate}
\item If $\msi \in \{42, 55, 69, 77\}$, then $a_5 A$ is a $\msp$-isolating class for any $\msp \in X \setminus L_{xy}$.
\item If $\msi \in \{66, 68, 81, 82\} = \IFi \setminus \{42, 55, 69, 77\}$, then $l A$ is a $\msp$-isolating class, where
\[
l = \begin{cases}
a_1 a_5, & \text{if $\msp \in H_x \setminus L_{xy}$}, \\
a_5, & \text{if $\msp \in X \setminus H_x$}.
\end{cases}
\]
\end{enumerate}
\end{Lem}

\begin{proof}
Suppose that either $\msi \in \{42, 55, 69, 77\}$ and $\msp \in X \setminus L_{xy}$ or $\msi \in \{66,68,81,82\}$ and $\msp \in X \setminus H_x$.
In this case we may assume $\msp = \msp_x$ after replacing coordinates.
Then $a_5 A$ is a $\msp$-isolating class by Lemma \ref{lem:findisol} (see also Remark \ref{rem:findisol}).

Suppose that $\msi \in \{66,68,81,82\}$ and $\msp \in H_x \setminus L_{xy}$.
Then $\msp \notin H_y$ and thus $a_1 a_5 A$ is a $\msp$-isolating class by Lemma \ref{lem:findisol}.
This completes the proof.
\end{proof}

\begin{Prop} \label{prop:nonsingpart2i}
Let $X$ be a member of family No.~$\msi$ with $\msi \in \IFi$.
Then $\lct_{\msp} (X) \ge 1$ for any nonsingular point $\msp \in X$.
\end{Prop}

\begin{proof}
By Proposition \ref{prop:compLxy}, it remains to consider $\msp$ with $\msp \notin L_{xy}$.
Recall that $a_1 a_5 (A^3) \le 1$ (see Remark \ref{rem:numcharact}).

Suppose that $\msi \in \{42, 55, 69, 77\}$.
In this case we have $1 = a_1 < a_2$.
Let $S := T_{\msp} \sim_{\mbQ} A$ be the unique member of the linear system $|\mcI_{\msp} (A)|$.
By Lemma \ref{lem:lcttang}, the pair $(X, S)$ is log canonical at $\msp$.
By Lemma \ref{lem:isol-i}, $a_5 A$ is a $\msp$-isolating class and we have $1 \cdot a_5 (A^3) = a_1 a_5 (A^3) \le 1$.
Thus $\lct_{\msp} (X) \ge 1$ by Lemma \ref{lem:exclG}.

Suppose that $\msi \in \{66, 68, 81, 82\}$.
In this case we have $2 \le a_1 < a_2$.
If $\msp \in H_x \setminus L_{xy}$, then we have $\lct_{\msp} (X) \ge 1$ by applying Lemma \ref{lem:exclG} for $S := H_x \sim_{\mbQ} A$ and the $\msp$-isolating class $a_1 a_5 A$ (see Lemma \ref{lem:isol-i}) since $a_1 a_5 (A^3) \le 1$.
We assume that $\msp \notin H_x$.
Let $S := T_{\msp} \sim_{\mbQ} a_1 A$ be the unique member of $|\mcI_{\msp} (a_1 A)|$.
By Lemma \ref{lem:multtang}, $(X,\frac{1}{a_1} S)$ is log canonical at $\msp$.
Thus, applying Lemma \ref{lem:exclG} for $S \sim_{\mbQ} a_1 A$ and the $\msp$-isolating class $a_5 A$ (see Lemma \ref{lem:isol-i}), we conclude $\lct_{\msp} (X) \ge 1$.
This completes the proof.
\end{proof}

\subsubsection{Families No.~$\msi$ with $\msi \in \IFii$}

\begin{Lem} \label{lem:lctnonsingptexcii}
Let $X$ be a member of family No.~$\msi$ with $\msi \in \IFii$.
Then $\lct_{\msp} (X) \ge 1$ for any nonsingular point $\msp \in \Upsilon_X$ of $X$.
\end{Lem}

\begin{proof}
This follows from Remark \ref{rem:numcharact}, Condition \ref{cdsp}(3) and Proposition \ref{prop:lctfflopcurve}.
\end{proof}

For the computation of LCT at a nonsingular point $\msp$ not contained in $L_{xy} \cup \Upsilon_X$, we determine a $\msp$-isolating class or a $(\msp, \Gamma)$-isolating class for a suitable curve $\Gamma$ and then apply Lemma \ref{lem:exclG} or \ref{lem:criwisol}.
We divide $\IFii$ into the disjoint union $\IFii = \IFiia \cup \IFiib \cup \IFiic$, where
\[
\begin{split}
\IFiia &= \{52, 63\}, \\ 
\IFiib &= \{40, 54, 61\}, \\
\IFiic &= \{43, 50, 53, 56, 57, 58, 62, 65, 67, 70, 72, 73, 74, 79, 80, 83\}.
\end{split}
\]

\begin{Lem} \label{lem:isol-iia}
Let $X$ be a member of family No.~$\msi$ with $\msi \in \IFiia = \{52, 63\}$ and $\msp$ a nonsingular point of $X$.
\begin{enumerate}
\item If $\msp \notin H_x \cup \Upsilon_X$, then $a_3 A$ is a $(\msp, \Gamma)$-isolating class, where $\Gamma \subset X$ is an irreducible and reduced curve with
\[
(A \cdot \Gamma) = \frac{a_4 + a_5}{a_4 a_5}, \quad
\mult_{\msp} (\Gamma) = 1.
\]
\item If $\msi = 52$ and $\msp \in H_x \setminus (L_{xy} \cup \Upsilon_X)$, then $a_1 a_3 A$ is a $\msp$-isolating class.
\item If $\msi = 63$ and $\msp \in H_x \setminus L_{xy}$, then either $15A$ is a $\msp$-isolating class or $5A$ is a $(\msp, \Gamma)$-isolating class, where $\Gamma \subset X$ is an irreducible and reduced curve with
\[
(A \cdot \Gamma) = \frac{5}{56}, \quad
\mult_{\msp} (\Gamma) = 1.
\]
\end{enumerate}
\end{Lem}

\begin{proof}
We first note that $\Upsilon_X = \Exc (\pi_t) \cup \Exc (\pi_u)$.
We prove (1).
Replacing coordinates we may assume $\msp = \msp_x$.
Set 
\[
\Xi := (y = z = s = 0) \cap X.
\]
Since $a_5 < d_1 < 2 a_4$ and $d_2 = a_4 + a_5 < 3 a_4$, we can write
\[
\begin{split}
F_1 (x,0,0,0,t,u) &= \alpha u x^{d_1-a_5} + \beta t x^{d_1-a_4}, \\
F_2 (x,0,0,0,t,u) &= u t + \gamma u x^{d_2-a_5} + \delta t^2 x^{d_2 - 2 a_4} + \varepsilon t x^{d_2 - a_4},
\end{split}
\]
for some $\alpha, \beta,\dots,\varepsilon \in \mbC$. 
We have $(\alpha,\gamma) \ne (0,0)$ since $\msp \notin \Exc (\pi_u)$.

Suppose that $\alpha \ne 0$.
Then we may assume $\alpha = 1$, and then by replacing $F_2$ with $F_2 - \gamma x^{d_2-d_1} F_1$ and replacing $u$, we may assume that
\[
\begin{split}
F_1 (x,0,0,0,t,u) &= u x^{d_1 - a_5}, \\
F_2 (x,0,0,0,t,u) &= u t + \delta t^2 x^{d_2-2 a_4} + \varepsilon t x^{d_2-a_4}.
\end{split}
\] 
We have $(\delta,\varepsilon) \ne (0,0)$ since $\msp \notin \Exc (\pi_t)$.
Thus $\Xi$ is a finite set of points and $a_3 A$ is a $\msp$-isolating class.

Suppose that $\alpha = 0$ and $\beta \ne 0$.
Then $\gamma \ne 0$ since $\msp \notin \Exc (\pi_u)$.
It is easy to see that $\Xi$ is a finite set of points and $a_3 A$ is a $\msp$-isolating class.

Suppose that $\alpha = \beta = 0$.
Then $\gamma \ne 0$ since $\msp \notin \Exc (\pi_u)$.
Rescaling $x$ and replacing $u$, we may assume $\gamma = 1$ and $\delta = 0$ so that
\[
F_2 (x,0,0,0,t,u) = u t + u x^{d_2-a_5} + \varepsilon t x^{d_2 - a_4}.
\]
We have $\varepsilon \ne 0$ since $\msp \notin \Exc (\pi_t)$.
We set
\[
\Gamma := \Xi = (y = z = s = u t + u x^{d_2-a_5} + \varepsilon t x^{d_2-a_4} = 0),
\]
which is an irreducible and reduced curve with $(A \cdot \Gamma) = (a_4+a_5)/a_4 a_5$ and $\mult_{\msp} (\Gamma) = 1$.
Thus $a_3 A$ is a $(\msp,\Gamma)$-isolating class and (1) is proved.

We prove (2).
We write $\msp = (0\!:\!1\!:\!\lambda\!:\!\mu\!:\!\nu\!:\!\theta)$ for some $\lambda, \mu, \nu, \theta \in \mbC$.

Suppose $\msi = 52$.
We may assume $\theta = 0$ by replacing $u \mapsto u - \theta y^4$.
Note that $(\lambda,\mu) \ne (0,0)$ since $\msp \notin \Exc (\pi_t)$.
If $\lambda \ne 0$ (resp.\ $\lambda = 0$ and $\mu \ne 0$), then $6A$ (resp.\ $10A$) is a $\msp$-isolating class by Lemma \ref{lem:isolstr}.

Suppose $\msi = 63$.
If $\lambda \ne 0$ (resp.\ $\lambda = 0$ and $\mu \ne 0$), then $12A$ (resp.\ $15A$) is a $\msp$-isolating class by Lemma \ref{lem:isolstr}.
Suppose that $\lambda = \mu = 0$.
%Then $\nu \ne 0$ and $\theta \ne 0$ since $\msp \notin \Exc (\pi_t) \cup \Exc (\pi_u)$.
By the quasi-smoothness of $X$, either $y^4 \in F_1$ or $y^5 \in F_2$.
But the condition $\msp \in X$ implies $y^4 \notin F_1$.
Thus $y^5 \in F_1$ and we can write
\[
F_2 (0,y,0,0,t,u) = t u - \nu \theta y^5,
\]
and we have $\nu \theta \ne 0$.
It follows that $\{x,z,s\}$ is a $(\msp,\Gamma)$-isolating set, where 
\[
\Gamma = (x = z = s = t u - \nu \theta y^5 = 0) \subset X
\] 
is clearly irreducible and reduced with $(A \cdot \Gamma) = 5/56$ and $\mult_{\msp} (\Gamma) = 1$.
This completes the proof. 
\end{proof}

\begin{Lem} \label{lem:isol-iibc1}
Let $X$ be a member of family No.~$\msi$ with $\msi \in \IFiib$ and $\msp \in X$ a nonsingular point of $X$.
Suppose that we are in one of the following cases.
\begin{enumerate}
\item $\msi \in \IFiib$ and $\msp \notin (L_{xy} \cup \Upsilon_X)$.
\item $\msi \in \IFiic$ and $\msp \notin (H_x \cup \Upsilon_X)$.
\end{enumerate}
Then $a_3 A$ is either a $\msp$-isolating class or a $(\msp,\Gamma)$-isolating class, where $\Gamma$ is an irreducible and reduced curve with $(A \cdot \Gamma) = 2/a_5$ and $\mult_{\msp} (\Gamma) = 1$.  
\end{Lem}

\begin{proof}
Let $\msp$ be as in the statement.
Then, by a choice of coordinates we may assume $\msp = \msp_x$.
We set 
\[
\Xi := (y = z = s = 0) \cap X.
\]
We will show that either $\Xi$ is a $\msp$-isolating set or it is an irreducible and reduced curve with $(A \cdot \Xi) = 2/a_5$ and $\mult_{\msp} (\Xi) = 1$.
By Lemma \ref{lem:isolsetclass}, in the former case $a_3 A$ is a $\msp$-isolating class and in the latter case $a_3 A$ is a $(\msp,\Gamma)$-isolating class by setting $\Gamma := \Xi$.

{\bf Case 1}: $\msi \in \{40,43,53,61,62,65,67,70,79,80,83\}$.
In this case we have $2 a_4 = d_1 < a_4 + a_5$, $d_2 < 3 a_4$, $d_2 < 2 a_5$ and $d_2 \le a_4 + a_5$.
Hence we can write
\[
\begin{split}
F_1 (x,0,0,0,t,u) &= \lambda u x^{d_1-a_5} + t^2 + \alpha t x^{d_1-a_4}, \\
F_2 (x,0,0,0,t,u) &= \mu u x^{d_2 - a_5} + \beta u t + \gamma t^2 x^{d_2 - 2 a_4} + \delta t x^{d_2 - a_4},
\end{split}
\]
for some $\lambda,\mu,\alpha,\beta,\gamma,\delta \in \mbC$.
We have $(\lambda,\mu) \ne (0,0)$ since $\msp \notin \Exc (\pi_u)$.

Suppose that $\lambda = 0$.
Then $\mu \ne 0$ and it is easy to see that $\Xi$ is a $\msp$-isolating set.

Suppose that $\lambda \ne 0$.
Replacing $F_2$ by $F_2 - \lambda^{-1} \mu x^{d_2-d_1} F_1$, we may assume that $\mu = 0$.
If $(\beta,\gamma,\delta) = (0,0,0)$, then 
\[
\Gamma := \Xi = (y = z = s = \lambda u x^{d_1-a_5} + t^2 + \alpha t x^{d_1-a_4} = 0)
\] 
is an irreducible and reduced curve with $(A \cdot \Gamma) = d_1/a_4 a_5 = 2/a_5$ and $\mult_{\msp} (\Gamma) = 1$.
If $(\beta,\gamma,\delta) \ne (0,0,0)$, then it is straightforward to see that $\Xi$ is a $\msp$-isolating set.

{\bf Case 2}: $\msi \in \{54,56,57,72,73,74\}$.
In this case $d_1 < d_2 = 2 a_4 < a_4 + a_5$ and we we can write
\[
\begin{split}
F_1 (x,0,0,0,t,u) &= \lambda u x^{d_1-a_5} + \alpha t x^{d_1 - a_4}, \\
F_2 (x,0,0,0,t,u) &= \mu u x^{d_2 - a_5} + t^2 + \beta t x^{d_2-a_4},
\end{split}
\]
for some $\lambda,\mu,\alpha,\beta \in \mbC$.
Since $\msp \notin \Exc (\pi_u)$, we have $(\lambda,\mu) \ne (0,0)$.

Suppose that $\lambda = 0$.
If in addition $\alpha = 0$, then 
\[
\Gamma := \Xi = (y = z = s = \mu u x^{d_2-a_5} + t^2 + \beta t x^{d_2-a_4} = 0)
\]
is an irreducible and reduced curve with $(A \cdot \Gamma) = d_2/a_4 a_5 = 2/a_5$ and $\mult_{\msp} (\Gamma) = 1$ since $\mu \ne 0$. 
If $\alpha \ne 0$, then $\Xi$ is clearly a $\msp$-isolating set since $\mu \ne 0$.

Suppose that $\lambda \ne 0$.
Then we may assume $\mu = 0$ by replacing $F_2$ with $F_2 - \lambda^{-1} \mu x^{d_2-d_1} F_1$ and it is straightforward to see that $\Xi$ is a $\msp$-isolating set.

{\bf Case 3}: $\msi \in \{50,58\}$.
Then we have $a_3 = 5, a_4 = a_5 = 7$ and $d_2 = 14$.
By choosing $u, t$ appropriately, we can write
\[
\begin{split}
F_1 (x,0,0,0,t,u) &= \lambda u x^{d_1-5} + \nu t x^{d_1-5}, \\
F_2 (x,0,0,0,t,u) &= u t + \mu u x^7 + \theta t x^7,
\end{split}
\]
for some $\lambda,\mu,\nu,\theta \in \mbC$. 
Note that $(\lambda,\mu) \ne (0,0)$ and $(\nu,\theta) \ne (0,0)$ since $\msp \notin \Exc (\pi_u)$ and $\msp \notin \Exc (\pi_t)$.

Suppose that $\lambda = 0$.
If $\nu = 0$, then $\mu \ne 0$ and $\theta \ne 0$ and 
\[
\Gamma := \Xi = (y = z = s = u t + \mu u x^7 + \theta t x^7 = 0)
\]
is an irreducible and reduced curve with $(A \cdot \Gamma) = 2/7 = 2/a_5$.
If $\nu \ne 0$, then $\Xi$ is $\msp$-finite since $\mu \ne 0$.

Suppose that $\lambda \ne 0$.
Replacing $F_2$ with $F_2 - \lambda^{-1} \mu x^{d_2-d_1} F_1$, we may assume $\mu = 0$.
Then it is straightforward to see that $\Xi$ is a $\msp$-isolating set since $(\nu,\theta) \ne (0,0)$.  
This completes the proof.
\end{proof}

\begin{Lem} \label{lem:isol-iic}
Let $X$ be a member of family No.~$\msi$ with $\msi \in \IFiic$ and $\msp \in H_x \setminus (L_{xy} \cup \Upsilon_X)$ a nonsingular point.
Then $l A$ is a $\msp$-isolating class for some $l$ with $l \le 1/(A^3)$.
\end{Lem}

\begin{proof}
{\bf Case 1}: $\msi \in \{53, 62, 72, 79, 80, 83\}$.
By Lemma \ref{lem:findisol}, we see that $l A$ is a $\msp$-isolating class, where
\[
l = \max_{k \in \{2,3,4\}} \{\lcm (a_1, a_k)\}.
\]
Suppose that $\msi \in \{53, 62, 72\}$.
Then we have $l = a_1 a_3 \le 1/(A^3)$.
Suppose that $\msi \in \{79,80,83\}$.
Then we have $l = a_1 a_4 \le 1/(A^3)$ by direct computations.
Thus the assertion follows in this case.

In the following we write $\msp = (0\!:\!1\!:\!\alpha\!:\!\beta\!:\!\gamma\!:\!\delta) \in X$, where $\alpha,\beta,\gamma,\delta \in \mbC$.
If the weight of a homogeneous coordinate $v \in \{z,s,t,u\}$ is even, then we may assume that the corresponding coordinate among $\{\alpha,\beta,\gamma,\delta\}$ is zero by replacing the coordinate $v$.
For example, if $a_3$ is even, then we may assume $\beta = 0$ by replacing $s \mapsto s- \beta y^{a_3/2}$.

{\bf Case 2}: $\msi \in \{43,50,56,57,65,74\}$.
In this case we have $a_1 = 2$ and $a_2 = 3$.
For a positive odd integer $e$, the semigroup $\langle 2,e\rangle \subset \mbZ$ generated by $2$ and $e$ contains the set $\mbZ_{\ge e-1}$.
It follows from Lemma \ref{lem:isolstr} that $l A$ isolates $\msp$, where
\[
l = \begin{cases}
\max \{6,a_5\} = a_5, & \text{if $\alpha \ne 0$}, \\
\max \{2 a_3, a_5\}, & \text{if $\beta \ne 0$ and $a_3$ is odd}.
\end{cases}
\]
Note that we are assuming that $\beta = 0$ if $a_3$ is even.
It remains to consider $\msp = (0\!:\!1\!:\!0\!:\!0\!:\!\gamma\!:\!\delta)$.
\begin{itemize}
\item Suppose that $\msi \in \{43,56,65\}$.
Then $a_5$ is even and hence $\msp = (0\!:\!1\!:\!0\!:\!0\!:\!\gamma\!:\!0)$.
In this case the set $(x = z = s = u = 0) \cap X$ is a finite set of points since $t^2$ is contained in $F_1$ or $F_2$.
Thus $a_5 A$ is a $\msp$-isolating class. 
\item Suppose that $\msi = 50$.
Then, by a suitable choice of $t, u$, we have
\[
\begin{split}
F_1 (0,y,0,0,t,u) &= \lambda y^5, \\
F_2 (0,y,0,0,t,u) &= t u + \mu y^7,
\end{split}
\]
for some $\lambda, \mu \in \mbC$.
We have $\lambda = 0$ since $\msp \in X$.
Then we have $\mu \ne 0$ since $X$ is quasi-smooth.
Note that if $\gamma = 0$, then we have $\msp = \msp_u$ since $F_2 (0,y,0,0,t,u) = t u + \mu y^7$ with $\mu \ne 0$.
Similarly, if $\delta = 0$, then $\msp = \msp_t$.
Hence $\gamma \delta \ne 0$ since $\msp$ is a nonsingular point of $X$.
It follows that the set
\[
\{x, z, s, \delta t - \gamma u \}
\]
is a $\msp$-isolating set.
Thus $a_5 A$ is a $\msp$-isolating class.
\item Suppose that $\msi \in \{57,74\}$.
In this case $a_1 = 2$, $a_5 = a_1 + a_4$, $d_2 = 2 a_4$, and we can write
\[
\begin{split}
F_1 (0,y,0,0,t,u) &= \lambda y^{d_1/2}, \\
F_2 (0,y,0,0,t,u) &= t^2 + \mu y^6,
\end{split}
\]
for some $\lambda, \mu \in \mbC$.
We have $\lambda = 0$ since $\msp \in X$, and then $\mu \ne 0$ since $X$ is quasi-smooth.
Note that $\gamma \ne 0$ since $F_2 (0,y,0,0,t,u) = t^2 + \mu y^{d_2/2}$.
We see that
\[
\{x, z, s, \gamma u - \delta t y \}
\]
is a $\msp$-isolating set.
It follows that $a_5 A$ is a $\msp$-isolating set.
\end{itemize}
As a summary of Case 2, for any nonsingular point $\msp \in H_x \setminus (L_x \cup \Upsilon)$, we conclude that $l A$ is a $\msp$-isolating class for some $l$ with $l \le \max \{ 2 a_3, a_5\}$.
It is straightforward to check that $\max \{ 2 a_3, a_5\} \le 1/(A^3)$.

{\bf Case 3}: $\msi \in \{58, 67, 70, 73\}$.
In this case we have $a_1 = 3, a_2 = 4, a_3 = 5$.
By Lemma \ref{lem:isolstr}, we can compute $\msp$-isolating class as follows: if $\alpha \ne 0$, then $12 A$ is a $\msp$-isolating class, and if $\beta \ne 0$, then $15A$ is a $\msp$-isolating class.
Thus $l A$ is a $\msp$-isolating class if either $\alpha \ne 0$ or $\beta \ne 0$.
We compute $\msp$-isolating classes when $\alpha = \beta = 0$.
We set $\Xi = (x = z = s = 0) \cap X$.
Note that $\msp \in \Xi$.

Suppose that $\msi \in \{58, 67, 73\}$.
If $\msi \in \{58, 73\}$, then $F_1 (0,y,0,0,t,u) = \lambda y^{d_1/3}$ for some $\lambda \in \mbC$.
If $\msi = 67$, then $F_2 (0,y,0,0,t,u) = \lambda y^5$ for some $\lambda \in \mbC$.
Note that $\lambda \ne 0$ by quasi-smoothness of $X$.
This implies that $\Xi$ is a finite set of points.
Hence $\Xi$ is a $\msp$-isolating set and $5 A$ is a $\msp$-isolating class.

Suppose that $\msi = 70$.
We can write
\[
\begin{split}
F_1 (0,y,0,0,t,u) &= \lambda u y + \mu t^2, \\
F_2 (0,y,0,0,t,u) &= \varepsilon t y^3,
\end{split}
\]
for some $\lambda, \mu, \varepsilon \in \mbC$.
Note that $\lambda \ne 0$ and $\mu \ne 0$ since $X$ is quasi-smooth.
If $\varepsilon \ne 0$, then $\Xi$ is a finite set of points and thus $5 A$ is a $\msp$-isolating class.
Suppose that $\varepsilon = 0$.
Recall that $\msp = (0\!:\!1\!:\!0\!:\!\gamma\!:\!\delta) \in \Xi$.
If one of $\gamma, \delta$ is $0$, then $\gamma = \delta = 0$, that is, $\msp = \msp_y$ is a singular point.
Thus $\gamma \delta \ne 0$.
Then the set
\[
\{x, z, s, \gamma \delta y^6 - t u \}
\]
is a $\msp$-isolating set.
It follows that $18A$ is a $\msp$-isolating class.
As a summary, if $\msi \in \{58, 67, 73\}$ (resp.\ $\msi = 70$), then $l A$ is a $\msp$-isolating class for some $l \le 15$ (resp.\ $l \le 18$) and we have $15 \le 1/(A^3)$ (resp.\ $18 \le 1/(A^3))$.
This completes the proof.
\end{proof}

\begin{Prop} \label{prop:nonsingpart2ii}
Let $X$ be a member of family No.~$\msi$ with $\msi \in \IFii$.
Then $\lct_{\msp} (X) \ge 1$ for any nonsingular point $\msp \in X$.
\end{Prop}

\begin{proof}
By Proposition \ref{prop:compLxy} and Lemma \ref{lem:lctnonsingptexcii}, it remains to prove $\lct_{\msp} (X) \ge 1$ for nonsingular points on $X \setminus (L_{xy} \cup \Upsilon_X)$.
Recall that $a_1 a_3 (A^3) \le 1$ and $a_1 a_5 (A^3) \le 2$ (see Remark \ref{rem:numcharact}).

{\bf Case 1}: $\msi \in \IFiia$.
Suppose that $\msp \in X \setminus (H_x \cup \Upsilon_X)$.
Set $S := T_{\msp} \sim_{\mbQ} a_1 A$.
The pair $(X, \frac{1}{a_1} S)$ is log canonical at $\msp$ by Lemma \ref{lem:multtang}.
Let $\Gamma$ be the curve as in Lemma \ref{lem:isol-iia}.(1).
Then $a_3 A$ is a $(\msp,\Gamma)$-isolating class and we have 
\[
a_3 \cdot \frac{a_4 + a_5}{a_4 a_5} > 1 \quad \text{and} \quad
a_1 a_3 (A^3) \le 1.
\]
Thus, by Lemma \ref{lem:criwisol}, we have $\lct_{\msp} (X) \ge 1$.

Suppose that $\msi = 52$ and $\msp \in H_x \setminus (L_{xy} \cup \Upsilon_X)$.
We set $S := H_x \sim_{\mbQ} A$.
The pair $(X, S)$ is log canonical at $\msp$ and $a_1 a_3 A$ is a $\msp$-isolating class by Lemma \ref{lem:isol-iia}.
Thus we can apply Lemma \ref{lem:exclG} and conclude that $\lct_{\msp} (X) \ge 1$ since $1 \cdot a_1 a_3 (A^3) \le 1$.

Suppose that $\msi = 63$ and $\msp \in H_x \setminus (L_{xy} \cup \Upsilon_X)$.
We set $S := H_x \sim_{\mbQ} A$.
The pair $(X, S)$ is log canonical at $\msp$.
If $15A$ is a $\msp$-isolating class, then we can apply Lemma \ref{lem:exclG} and conclude that $\lct_{\msp} (X) \ge 1$ since $1 \cdot 15 (A^3) = 45/56 < 1$.
Otherwise $5A$ is a $(\msp,\Gamma)$-isolating class, where $\Gamma$ is the curve given in Lemma \ref{lem:isol-iia}(3).
We apply Lemma \ref{lem:criwisol} and conclude $\lct_{\msp} (X) \ge 1$ since 
\[
5 \cdot \frac{5}{56} = \frac{25}{56} < 1 \quad \text{and} \quad
1 \cdot 1 \cdot (A^3) = \frac{3}{56} < \frac{5}{56}.
\]

{\bf Case 2}: $\msi \in \IFiib$.
Suppose that $\msp \in X \setminus (L_{xy} \cup \Upsilon_X)$.
Set $S := T_{\msp} \sim_{\mbQ} a_1 A$.
The pair $(X, \frac{1}{a_1} S)$ is log canonical at $\msp$ by Lemma \ref{lem:lcttang}.
By Lemma \ref{lem:isol-iibc1}, either $a_3 A$ is a $\msp$-isolating class or a $(\msp, \Gamma)$-isolating class for some irreducible and reduced curve $\Gamma$ with $(A \cdot \Gamma) = 2/a_5$ and $\mult_{\msp} (\Gamma) = 1$.
In the former case, we have $\lct_{\msp} (X) \ge 1$ by Lemma \ref{lem:exclG} since $a_1 a_3 (A^3) \le 1$.
In the latter case, we have $\lct_{\msp} (X) \ge 1$ by Lemma \ref{lem:criwisol} since
\[
a_3 \cdot \frac{2}{a_5} \le 1 \quad \text{and} \quad a_1 \cdot 1 \cdot (A^3) \le \frac{2}{a_5}.
\]

{\bf Case 3}: $\msi \in \IFiic$.
Suppose that $\msp \in X \setminus (H_x \cup \Upsilon_X)$.
Then, in view of Lemma \ref{lem:isol-iibc1}, we can follow the same argument as in Case 2 and we conclude $\lct_{\msp} (X) \ge 1$ by applying Lemma \ref{lem:exclG} or \ref{lem:criwisol} for $S := T_{\msp} \sim_{\mbQ} a_1 A$.

Suppose that $\msp \in H_x \setminus (L_{xy} \cup \Upsilon_X)$.
We set $S := H_x \sim_{\mbQ} A$.
By Lemma \ref{lem:isol-iic}, $l A$ is a $\msp$-isolating class for some $l \le 1/(A^3)$.
Then we have $\lct_{\msp} (A) \ge 1$ by Lemma \ref{lem:exclG} since $1 \cdot l \cdot (A^3) \le 1$.
This completes the proof.
\end{proof}

%%%%%%%%%%%%%%%%%%%%%%%%%%%%%%%%%%%%
\subsection{Singular points}
%%%%%%%%%%%%%%%%%%%%%%%%%%%%%%%%%%%%

\begin{Lem} \label{lem:lctsingpt}
Let $X$ be a member of family No.~$\msi$ with $\msi \in \IFi \cup \IFii$ and $\msp \in X$ a singular point which is not marked $\QI, \EI$ or $\mathrm{d}$ as a subscript in \emph{Tables \ref{table:codim2Fanos-i}} and \emph{\ref{table:codim2Fanos-ii}}.
Then $\lct_{\msp} (X) \ge 1$.
\end{Lem}

\begin{proof}
Let $\msp \in X$ be as in the statement and let $\varphi \colon Y \to X$ be the Kawamata blowup at $\msp$.
Then $(B^3) \le 0$, where $B = -K_Y$.
It is proved in \cite[Section 8]{Okada1} that the condition $(B^2) \notin \bNE (Y)$ is satisfied for any singular point $\msp \in X$ with $(B^3) \le0$.
It is also proved that in \cite[Section 8]{Okada1} that there exists a prime divisor $S \in m B$ on $X$ for some $m \ge 1$ such that $\tilde{S} \sim_{\mbQ} m B$, where $\tilde{S}$ is the proper transform of $S$ by the Kawamata blowup at $\msp$.
This can also be seen as follows: Suppose that $\msp \in X$ is of type $\frac{1}{r} (1,a,r-a)$, then there exists a homogeneous coordinate $x_j$ whose weight $a_j$ is less than $r$ and coincides with one of $1, a, r-a$.
In this case the divisor $H_{x_j} \sim_{\mbQ} a_j A$ satisfies $\tilde{H}_{x_j} \sim_{\mbQ} a_j B$.
Therefore we can apply Lemma \ref{lem:singptNE} and conclude that $\lct_{\msp} (X) \ge 1$.
\end{proof}
 
\begin{Prop} \label{prop:lctsingpI}
Let $X$ be a member of family No.~$\msi$ with $\msi \in \IFi$.
Then $\lct_{\msp} (X) \ge 1$ for any singular point $\msp \in X$.
\end{Prop}

\begin{proof}
By Lemma \ref{lem:lctsingpt}, it remains to consider a singular point $\msp \in X$ which is marked $\QI$, $\EI$ or $\mathrm{d}$ as a subscript in Table \ref{table:codim2Fanos-i}.
Let $\msp$ be such a point and suppose that it is of type $\frac{1}{r} (1,a,r-a)$.
Note that $\msp \in L_{xy}$.
Let $S_1 \in |A|$ be a general member so that it is quasi-smooth and thus $(X, S_1)$ is log canonical at $\msp$.
Let $S_2 \in |a_1 A|$ be a general member so that $S_1 \cdot S_2 = L_{xy}$.
Since $a_1 a_5 (A^3) \le 1$ (see Remark \ref{rem:numcharact}) and $r \le a_5$, we have $r \cdot 1 \cdot a_1 (A^3) \le a_5 a_1 (A^3) \le 1$.
Therefore, by the description of singularities given in Table \ref{table:Lxy-i} and by Remark \ref{rem:multLxy}, we can apply Lemma \ref{lem:exclL} and conclude that $\lct_{\msp} (X) \ge 1$.
\end{proof}

\begin{Prop} \label{prop:singmost}
Let $X$ be a member of family No.~$\msi$ with $\msi \in \IFii$.
Then the following assertions hold.
\begin{enumerate}
\item $\lct_{\msp} (X) \ge 1$ for any singular point $\msp \in X$ except possibly for distinguished singular points.
\item If $\msi \in \{79, 80, 83\}$, then $\lct_{\msp} (X) \ge 1$ for any singular point $\msp \in X$.
\end{enumerate}
\end{Prop}

\begin{proof}
Let $\msp$ be a singular point marked $\QI$ or $\EI$.
Suppose that $\msp \in X$ is of type $\frac{1}{r} (1,a,r-a)$.
Then, by Table \ref{table:Lxy-ii}, we see that $r \le a_3$ (in fact $r = a_3$ except when $\msi = 40$ and in that case $r = a_2 < a_3$).
Moreover $\msp \in L_{xy}$.
Let $S_1 \in |A|$ be a general member so that it is quasi-smooth and thus $(X, S_1)$ is log canonical at $\msp$.
Let $S_2 \in |a_1 A|$ be a general member so that $S_1 \cdot S_2 = L_{xy}$.
We have $r \cdot 1 \cdot a_1 (A^3) \le a_1 a_3 (A^3) \le 1$ (see Remark \ref{rem:numcharact}).
By the description of singularities of $L_{xy}$ given in Table \ref{table:Lxy-ii}, and by Remark \ref{rem:multLxy}, we can apply Lemma \ref{lem:exclL} and conclude that $\lct_{\msp} (X) \ge 1$.
In view of Lemma \ref{lem:lctsingpt}, this proves (1).

We prove (2).
By the above argument, it remains to consider the distinguished singular point $\msp \in X$ for $\msi \in \{79,80,83\}$.
In the notation of Lemma \ref{lem:somedistsingpt} (or of Section \ref{sec:flop}), we have $a_k = a_5$.
Moreover $a_{j_2} = 6, 7, 4$ if $\msi = 79, 80, 83$, respectively.
It is straightforward to compute that $a_k a_{j_2} (A^3) = 2$ for $\msi \in \{79, 80, 83\}$ and $H_x \cap \Upsilon_{\msp} = \{\msp\}$ (see Example \ref{ex:Upsilon} below).
Thus, by Lemma \ref{lem:somedistsingpt}, we have $\lct_{\msp} (X) \ge 1$.
\end{proof}

\begin{Ex} \label{ex:Upsilon}
We observe that $H_x \cap \Upsilon_{\msp} = \{\msp\}$ for the distinguished singular point of a member of family No.~$\msi$ with $\msi \in \{79,80,83\}$.
We use notation in Section \ref{sec:flop}.
We have $\msp = \msp_u$ and $x_k = u$, i.e.\ $a_5 = a_5$.
We see that $H_x \cap \Upsilon_{\msp} = \{\msp\}$ if and only if the set
\[
\begin{split} &(x = 0) \cap \overline{\Upsilon}_{\msp} \subset \mbP (1, a_{i_1}, a_{i_2}), \\
\cong & \ \Xi_{\msp} := (F_1|_{x = x_{j_1} = x_{j_2} = x_k = 0} = F_2|_{x = x_{j_1} = x_{j_2} = x_k = 0} = 0) \subset \mbP (a_{i_1}, a_{i_2}),
\end{split}
\]
is empty.

Suppose that $\msi = 79$.
Then $X = X_{18,20} \subset \mbP (1,4,5,6,9,14)$.
Then $x_{i_1} = z, x_{i_2} = t$, i.e.\ $a_{i_1} = 5, a_{i_2} = 9$, and we have
\[
\Xi_{\msp} = (F_1 (0,0,z,0,t,0) = F_2 (0,0,z,0,t,0) = 0) \subset \mbP (5,9).
\]
We have $F_1 (0,0,z,0,t,0) = t^2$ and $F_2 (0,0,z,0,t,0) = z^4$ since $X$ is quasi-smooth.
It is then easy to see that $\Xi_{\msp} = \emptyset$.

If $\msi = 80$, then, by a similar argument, we have
\[
\Xi_{\msp} = (F_1 (0,y,0,0,t,0) = F_2 (0,y,0,0,t,0) = 0) \subset \mbP (4,9)
\]
and $F_1 (0,y,0,0,t,0) = t^2$, $F_2 (0,y,0,0,t,0) = y^5$.

If $\msi = 83$, then
\[
\Xi_{\msp} = (F_1 (0,0,0,s,t,0) = F_2 (0,0,0,s,t,0) = 0) \subset \mbP (7,10)
\]
and $F_1 (0,0,0,s,t,0) = t^2$, $F_2 (0,0,0,s,t,0) = s^3$.
Thus $\Xi_{\msp} = \emptyset$ in both cases.
\end{Ex}

Theorems \ref{thm:main1} and \ref{thm:main2} follow from Propositions \ref{prop:nonsingpart2i}, \ref{prop:nonsingpart2ii}, \ref{prop:lctsingpI} and \ref{prop:singmost}.

%%%%%%%%%%%%%%%%%%%%%%%%%%%%%%%%%%
\section{Examples of super-rigid affine $4$-folds} \label{sec:ex}
%%%%%%%%%%%%%%%%%%%%%%%%%%%%%%%%%%

Let $Y$ be a normal projective $\mbQ$-factorial variety of Picard number $1$ and let $S$ be a prime divisor on $Y$ such that the pair $(Y,S)$ is plt (see \cite[Definition 2.34]{KM} for the definition of plt). 
Note that $Y \setminus S$ is an affine variety since $S$ is ample, and we say that $Y \setminus S$ is an {\it affine Fano variety} with completion $Y$ and boundary $S$ if $-(K_Y +S)$ is ample.
Following is the definition of super-rigid affine Fano variety, where we refer readers to \cite{CDP} for the detail of the condition (2) below.

\begin{Def}[{\cite[Definition 1.4]{CDP}}]
An affine Fano variety $(Y, S)$ is said to be {\it super-rigid} if the following conditions hold.
\begin{enumerate}
\item For every affine Fano variety $Y' \setminus S'$ with completion $Y'$ and boundary $S'$, if there exists an isomorphism $\phi \colon Y \setminus S \to Y' \setminus S'$, then $\phi$ is induced by an isomorphism $Y \cong Y'$ that maps $S$ onto $S'$.
\item The affine Fano variety $Y \setminus S$ does not contain relative affine Fano varieties of the same dimension over varieties of positive dimensions.
\end{enumerate}
\end{Def}

It is worthwhile to mention that, for a super-rigid affine Fano varieties $Y \setminus S$, we have $\Aut (Y \setminus S) = \Aut (S, Y)$ and $Y \setminus S$ does not contain an open $\mbA^1$-cylinder, that is, it does not contain a Zariski open subset which is isomorphic to $Z \times \mbA^1$ for some quasi-projective varieties.

As an application of Theorem \ref{thm:main1}, we have the following examples.
In the following statement, a $\mbQ$-Fano $4$-fold means a normal projective $\mbQ$-factorial Fano variety with at most log terminal singularities.

\begin{Thm}
Let $(d; a_0,a_1,a_2,a_3,a_4,a_5)_{\msi}$ be one of the following septuple.
\begin{longtable}{lll}
$(14;1,2,5,6,7,9)_{66}$ & $(15;1,2,5,6,7,9)_{66}$ & $(15;1,3,5,6,7,8)_{68}$ \\
$(16;1,1,5,7,8,9)_{69}$ & $(18;1,1,6,8,9,10)_{77}$ & $(22;1,2,5,9,11,13)_{82}$ \\ 
$(21;1,3,4,7,10,17)_{83}$ & &
\end{longtable}
Let $d'$ be the positive integer such that a general weighted complete intersection of type $(d, d')$ in $\mbP (a_0,\dots,a_5)$ belongs to family No.~$\msi$.
Let $Y$ be a general weighted hypersurface of degree $d$ in $\mbP (a_0,\dots,a_5)$ and $X \subset Y$ be the hypersurface in $Y$ defined by a general homogeneous polynomial of degree $d'$.
Then $Y$ is a quasi-smooth and well formed $\mbQ$-Fano $4$-fold of Picard number $1$ and the complement $Y \setminus X$ is an affine super-rigid Fano $4$-fold with completion $Y$ and boundary $X$.
\end{Thm}

%\begin{table}[htb]
%\begin{center}
%\caption{Affine super-rigid Fano $4$-folds}
%\label{table:affineFano}
%\begin{tabular}{ccc}
%$(d; a_0,a_1,a_2,a_3,a_4,a_5)$ & No. & $\Sing (Y)$ \\
%\hline \\[-3mm]
%$(14;1,2,5,6,7,9)$ & 66 & $\frac{1}{2}, \frac{1}{3} (1,1,2,2), \frac{1}{5} (1,1,2,2), \frac{1}{6} (1,1,3,5), \frac{1}{9} (1,2,6,7)$ \\[1mm]
%$(15;1,2,5,6,7,9)$ & 66 & $\frac{1}{2}, \frac{1}{3} (1,1,2,2), \frac{1}{6} (1,1,2,5), \frac{1}{7} (2,2,5,6), \frac{1}{9} (1,2,5,7)$ \\[1mm]
%$(15;1,3,5,6,7,8)$ & 68 & $\frac{1}{3} (1,1,2,2), \frac{1}{6} (1,1,2,5), \frac{1}{7} (1,3,5,6), \frac{1}{8} (1,3,5,6)$ \\[1mm]
%$(16;1,1,5,7,8,9)$ & 69 & $\frac{1}{5} (1,2,3,4), \frac{1}{7} (1,1,1,5), \frac{1}{9} (1,1,5,8)$ \\[1mm]
%$(18;1,1,6,8,9,10)$ & 77 & $\frac{1}{2} (0,1,1,1), \frac{1}{3} (1,1,1,2), \frac{1}{8} (1,1,1,6), \frac{1}{10} (1,1,6,9)$ \\[1mm]
%$(22;1,2,5,9,11,13)$ & 82 & $\frac{1}{5} (1,1,3,4), \frac{1}{9} (1,2,2,5), \frac{1}{13} (1,2,5,11)$\\[1mm]
%$(21;1,3,4,7,10,17)$ & 83 & $\frac{1}{2}, \frac{1}{4} (1,2,3,3), \frac{1}{10} (3,4,7,7), \frac{1}{17} (1,3,7,10)$ \\
%\end{tabular}
%\end{center}
%\end{table}

\begin{proof}
It is straightforward to see that $Y$ is quasi-smooth and well formed by the criteria given in \cite{IF}.
It is obvious that $Y$ is $\mbQ$-Fano because $-K_Y$ is ample by adjunction and $Y$ has only cyclic quotient singularities.
Note that $\Diff_X (0) \ne 0$ only if there exists a component of $\Sing (X)$ of codimension $2$ in $Y$ that is contained in $X$.
The latter is impossible since both $X$ and $Y$ are quasi-smooth and well formed.
Thus 
\[
\alpha (X, \Diff_X (0)) =  \alpha (X,0) = \lct (X) = 1,
\]
where the last equality follows from Theorem \ref{thm:main1} since $X$ is a general member of family No.~$\msi$ with $\msi \in \IFi \cup \{79,80,83\}$, and we conclude from \cite[Theorem B]{CDP} that $Y \setminus X$ is an affine super-rigid Fano $4$-fold.
\end{proof}

%%%%%%%%%%%%%%%%%%%%%%%%%%%%%%%%%%

\end{document}